\documentclass[reqno,10pt, centertags]{amsart}
\usepackage{amssymb,upref,esint,color}
\usepackage{hyperref}

\newcommand*{\mailto}[1]{\href{mailto:#1}{\nolinkurl{#1}}}

\makeatletter
\def\theequation{\@arabic\c@equation}

\newcommand{\bbN}{{\mathbb{N}}}
\newcommand{\bbR}{{\mathbb{R}}}

\newcommand{\bbZ}{{\mathbb{Z}}}
\newcommand{\bbC}{{\mathbb{C}}}

\newcommand{\cB}{{\mathcal B}}
\newcommand{\cC}{{\mathcal C}}
\newcommand{\cD}{{\mathcal D}}
\newcommand{\cE}{{\mathcal E}}
\newcommand{\cF}{{\mathcal F}}
\newcommand{\cH}{{\mathcal H}}

\newcommand{\cK}{{\mathcal K}}

\newcommand{\cN}{{\mathcal N}}

\newcommand{\dott}{\,\cdot\,}

\newcommand{\no}{\nonumber}
\newcommand{\lb}{\label}
\newcommand{\f}{\frac}
\newcommand{\ul}{\underline}
\newcommand{\ol}{\overline}

\newcommand{\wti}{\widetilde}

\newcommand{\Oh}{O}

\newcommand{\tr}{\text{\rm{tr}}}

\newcommand{\ran}{\text{\rm{ran}}}
\newcommand{\ind}{\text{\rm{ind}}}
\newcommand{\de}{\text{\rm{def}}}
\newcommand{\dom}{\text{\rm{dom}}}

\newcommand{\bi}{\bibitem}

\numberwithin{equation}{section}

\newtheorem{theorem}{Theorem}[section]
\newtheorem{lemma}[theorem]{Lemma}
\newtheorem{corollary}[theorem]{Corollary}
\theoremstyle{definition}
\newtheorem{definition}[theorem]{Definition}
\newtheorem{hypothesis}[theorem]{Hypothesis}
\newtheorem{remark}[theorem]{Remark}
\newtheorem{example}[theorem]{Example}

\begin{document}

\title[Factorizations of Analytic Operator-Valued Functions]{On Factorizations of Analytic 
Operator-Valued Functions and Eigenvalue Multiplicity Questions}

\author[F.\ Gesztesy]{Fritz Gesztesy} 
\address{Department of Mathematics,
University of Missouri, Columbia, MO 65211, USA}
\email{\mailto{gesztesyf@missouri.edu}}
\urladdr{\url{http://www.math.missouri.edu/personnel/faculty/gesztesyf.html}}

\author[H.\ Holden]{Helge Holden}
\address{Department of Mathematical Sciences,
Norwegian University of
Science and Technology, NO--7491 Trondheim, Norway}
\email{\mailto{holden@math.ntnu.no}}
\urladdr{\href{http://www.math.ntnu.no/~holden/}{http://www.math.ntnu.no/\~{}holden/}}

\author[R.\ Nichols]{Roger Nichols}
\address{Mathematics Department, The University of Tennessee at Chattanooga, 
415 EMCS Building, Dept. 6956, 615 McCallie Ave, Chattanooga, TN 37403, USA}
\email{\mailto{Roger-Nichols@utc.edu}}
\urladdr{\url{http://www.utc.edu/faculty/roger-nichols/index.php}}

\date{\today}
\thanks{Supported in part by the Research Council of Norway.} 
\thanks{R.N. gratefully acknowledges support from an AMS--Simons Travel Grant.} 
\thanks{{\it Integral Eq. Operator Theory} {\bf 82}, 61--94 (2015), Erratum {\bf 85}, 301--302 (2016).}

\begin{abstract}
We study several natural multiplicity questions that arise in the context of the 
Birman--Schwinger principle applied to non-self-adjoint operators. In particular, we re-prove 
(and extend) a recent result by Latushkin and Sukhtyaev by employing a different technique 
based on factorizations of analytic operator-valued functions due to Howland. Factorizations 
of analytic operator-valued functions are of particular interest in themselves and again we 
re-derive Howland's results and subsequently extend them.    

Considering algebraic multiplicities of finitely meromorphic operator-valued functions, we 
recall the notion of the index of a finitely meromorphic operator-valued function and use  
that to prove an analog of the well-known Weinstein--Aronszajn formula relating algebraic 
multiplicities of the underlying unperturbed and perturbed operators.  

Finally, we consider pairs of projections for which the difference belongs to the trace class 
and relate their Fredholm index to the index of the naturally underlying Birman--Schwinger 
operator.
\end{abstract}

\maketitle

{\scriptsize{\tableofcontents}}

\section{Introduction} \lb{s1}

One of the principal aims of this paper is to investigate factorizations of analytic operator-valued functions and their connection to eigenvalue multiplicity questions for perturbations of a non-self-adjoint densely defined, closed linear operator $H_0$ in a separable Hilbert space $\cH$ by a 
non-self-adjoint additive perturbation term which formally factors into a product $V_2^*V_1$ 
over an auxiliary Hilbert space, $\cK$:
\begin{equation}\lb{1.1}
V_j:\dom(V_j)\subseteq \cH\rightarrow \cK,\quad j\in \{1,2\},
\end{equation}
with $V_j$ densely defined, closed linear operators.  Under appropriate assumptions on $V_j$, 
$j=1,2$, Kato \cite{Ka66} introduced a technique to define the ``sum,'' $H=H_0 + V_2^*V_1$ indirectly in 
terms of its resolvent $R(z)=(H-zI_{\cH})^{-1}$ and the free resolvent $R_0(z)=(H_0-zI_{\cH})^{-1}$, according to the equation, 
\begin{equation}\lb{1.2}
\begin{split} 
R(z) = R_0(z) - \overline{R_0(z)V_2^*}\big[I_{\cK} - \ol{V_1 R_0(z) V_2^*}\big]^{-1}V_1R_0(z),&\\
z\in \{\zeta\in \rho(H_0)\, |\, 1\in \rho(K(\zeta))\},&
\end{split}
\end{equation}
where $K(\cdot)$ represents an abstract Birman--Schwinger-type operator defined by
\begin{equation}\lb{1.3}
K(z) = -\ol{V_1 R_0(z)V_2^*}, \quad z \in \rho(H_0).
\end{equation}
Kato assumed $H_0$ to be self-adjoint and shows that under appropriate hypotheses on $V_j$, $j=1,2$, \eqref{1.2} defines the resolvent of a densely defined, closed operator $H$ in $\cH$. 

Shortly thereafter, assuming $K(z)$ to be compact for each $z\in \rho(H_0)$, Konno and Kuroda \cite{KK66} established an abstract variant of the Birman--Schwinger principle by proving that 
$z_0\in \rho(H_0)$ is an eigenvalue of $H$ if and only if $1$ is an eigenvalue of the Birman--Schwinger operator $K(z_0)$.  Moreover, one has equality of the corresponding {\it geometric multiplicities} (i.e., the dimensions of the corresponding eigenspaces):  the geometric multiplicity of $z_0$ as an eigenvalue of $H$ coincides with the geometric multiplicity of $1$ as an eigenvalue of $K(z_0)$.  Actually, Konno and Kuroda \cite{KK66} assume that $H_0$ is self-adjoint and 
\begin{equation}\lb{1.4}
(V_1f,V_2g)_{\cK} = (V_2f,V_1g)_{\cK}, \, \text{ for all } \, f,g\in \dom(V_1)\cap\dom(V_2),
\end{equation}
which guarantees that $H$ is self-adjoint as well. However, as shown in \cite{GLMZ05}, the 
construction of $H$ by Kato \cite{Ka66} via \eqref{1.2} and the geometric multiplicity results by Konno and Kuroda \cite{KK66} extend to the case where $H_0$ and $H$ are non-self-adjoint, and, moreover, the compactness assumption on $K(z)$ may be relaxed and replaced by the assumption that $I_{\cK}-K(z)$ is a Fredholm operator in $\cK$ for each $z\in \rho(H_0)$.

The abstract formulation of the Birman--Schwinger principle set forth by Konno and Kuroda  \cite{KK66} (and its extension in \cite{GLMZ05}) yields equality of the geometric multiplicities of a point $z_0\in \rho(H_0)$ as an eigenvalue of $H$ and $1$ as an eigenvalue of the Birman--Schwinger operator $K(z_0)$. Thus, if the operators $H$ and $K(z_0)$ are non-self-adjoint, the algebraic and geometric multiplicities of $z_0$ (resp., $1$) as an eigenvalue of $H$ 
(resp., $K(z_0)$) will differ in general and hence it is entirely natural to inquire about the status 
of the associated {\it algebraic multiplicities} (defined in terms of the dimension of the range of 
the associated Riesz projections).

At first, one might be tempted to conjecture that the algebraic multiplicities of $z_0$ as an eigenvalue of $H$ and $1$ as an eigenvalue of $K(z_0)$ coincide, as with the corresponding geometric multiplicities. However, this is easily dismissed by explicit counterexamples. Therefore, one is forced to search for alternatives. It turns out that in lieu of the algebraic multiplicity of $1$ as an eigenvalue of $K(z_0)$, one instead should consider the algebraic multiplicity of $z_0$ as a 
zero of finite-type of the analytic operator-valued function $I_{\cK}-K(\cdot)$, denoted by $m_a(z_0;I_{\cK}-K(\cdot))$.  Indeed, Latushkin and Sukhtyaev \cite{LS10} have shown, under appropriate assumptions, that the algebraic multiplicity of $z_0$ as an eigenvalue of $H$ coincides with the algebraic multiplicity of $z_0$ as a zero of finite-type of the analytic operator-valued function $I_{\cK}-K(\cdot)$.

Latushkin and Sukhtyaev rely on a Gohberg--Sigal Rouche-type Theorem to arrive at their result.  We recover and slightly extend their result by entirely different means. The approach employed in this paper relies on a factorization technique for analytic operator-valued functions originally due to Howland \cite{Ho71} (of interest in its own right), which shows that, given an analytic family 
$A(\cdot)$ of compact operators 
in $\cH$ defined on a domain $\Omega$ with $z_0\in \Omega$ a weak zero (i.e., $A(z_0)$ is not invertible), if $P$ is any projection onto $\ran(A(z_0))$, and $Q=I_{\cH}-P$, then $A(\cdot)$ may 
be factored according to
\begin{equation}\lb{1.5}
A(z) = [Q - (z-z_0)P]A_1(z),\quad z\in \Omega,
\end{equation}
where $A_1(z)$ is analytic in $\Omega$, $A_1(z)-I_{\cH}$ is compact, and 
\begin{equation}\lb{1.6}
\dim\big(\ran(A_1(z_0))^\perp\big) \leq \dim\big(\ran(A(z_0))^\perp \big).
\end{equation}
Howland's motivation for arriving at such a factorization was to determine necessary and sufficient conditions for a pole of $A(z)^{-1}$ to be simple.  In this paper, we extend Howland's factorization result to analytic families of Fredholm operators, thus relaxing the compactness assumption on $A(\cdot)$, and apply the factorization to algebraic multiplicities within the context of the abstract Birman--Schwinger principle. We emphasize that these considerations, combined with Evans function methods (cf.\ \cite{BLR14}, \cite{GLM07}, \cite{GLZ08} and the extensive literature cited therein), have immediate applications in the area of linear stability theory for nonlinear evolution equations. 

Next, we briefly turn to a summary of the contents of this paper:  In Section \ref{s2}, we recall the method introduced by Kato \cite{Ka66}, and extended in \cite{GLMZ05}, to define additive perturbations $H=H_0+W$ of a non-self-adjoint operator $H_0$ by a non-self-adjoint perturbation term which formally factors according to $W=V_2^*V_1$.  More specifically, the sum $H=H_0+W$ is defined indirectly through a resolvent formalism, with the resolvent of $H$ given by \eqref{1.2}.
Introducing the Birman--Schwinger operator $K(\cdot)$ in \eqref{1.3} and assuming $I_{\cH}-K(z)$ is a Fredholm operator for all $z\in \rho(H_0)$, we recall in Theorem \ref{2.7} an abstract version (of a variant) of the Birman--Schwinger principle due to Konno and Kuroda \cite{KK66} in the case where $H_0$ and $H$ are self-adjoint and $K(z)$ is compact, that states $z_0\in \rho(H_0)$ is an eigenvalue of $H$ if and only if $1$ is an eigenvalue of $K(z_0)$, and the geometric multiplicity of $z_0$ as an eigenvalue of $H$ is finite and coincides with the geometric multiplicity of $1$ as an eigenvalue of $K(z_0)$.  In Section \ref{s3}, we recall the notion of a finitely meromorphic family of operators and the analytic Fredholm theorem in Theorem \ref{t3.3}.  In Theorems \ref{t3.4} and \ref{t3.5}, we extend Howland's factorization to the case of an analytic family $A(\cdot)$ of Fredholm operators and recover Howland's necessary and sufficient condition for a pole of $A(\cdot)^{-1}$ to be simple in Corollary \ref{c3.6}.  Theorems \ref{t3.7} and \ref{t3.8} provide factorizations analogous to those in Theorems \ref{t3.4} and \ref{t3.5} but with the orders of the factors reversed (this appears to be a new result). In Section \ref{s4}, we consider algebraic multiplicities of zeros of analytic operator-valued functions and study their application to the abstract Birman--Schwinger operator $K(\cdot)$.  In Theorem \ref{t3.11}\footnote{We slightly corrected and extended the first paragraph of the proof of Theorem \ref{t3.11}.}, we use the extension of Howland's factorization to reprove and slightly extend the algebraic multiplicity result of Latushkin and Sukhtyaev \cite{LS10}, proving that any $z_0\in \rho(H_0)\cap \sigma(H)$ is a discrete eigenvalue of $H$ if $z_0$ is isolated in $\sigma(H)$.  In this case, $z_0$ is a zero of finite algebraic multiplicity of $I_{\cK}-K(\cdot)$, and the algebraic multiplicity of $z_0$ as an eigenvalue of $H$ equals the algebraic multiplicity of $z_0$ as a zero of $I_{\cK}-A(\cdot)$.  Example \ref{e3.13} shows that the algebraic multiplicity of $z_0$ as an eigenvalue of $H$ need not equal the multiplicity of $1$ as an eigenvalue of $K(z_0)$ in general.  In Section \ref{s5}, we extend Theorem \ref{t3.11} to the case where $K(\cdot)$ is finitely meromorphic and recover an analog of the Weinstein--Aronszajn formula for the case when $H_0$ and $H$ have common discrete eigenvalues.  In our final Section \ref{s6}, we apply some of the results from Sections \ref{s4} and \ref{s5} to ordered pairs of projections $(P,Q)$ for which the difference $P-Q$ belongs to the trace class.

We will use the following notation in this paper. Let $\cH$ and $\cK$
be separable complex Hilbert spaces, $(\cdot,\cdot)_{\cH}$ and
$(\cdot,\cdot)_{\cK}$ the scalar products in $\cH$ and $\cK$ (linear in
the second factor), and
$I_{\cH}$ and $I_{\cK}$ the identity operators in $\cH$ and $\cK$,
respectively. Next, let $T$ be a closed linear operator from
$\dom(T)\subseteq\cH$ to $\ran(T)\subseteq\cK$, with $\dom(T)$
and $\ran(T)$ denoting the domain and range of $T$. The closure of a
closable operator $S$ is denoted by $\ol S$. The kernel (null space) of $T$
is denoted by $\ker(T)$. The spectrum, point spectrum (i.e., the set of 
eigenvalues), and resolvent set of a closed linear operator in $\cH$ will be 
denoted by $\sigma(\cdot)$, $\sigma_p(\cdot)$, and $\rho(\cdot)$; the 
discrete spectrum of $T$ (i.e., points in $\sigma_p(T)$ which are isolated from the rest of 
$\sigma(T)$, and which are eigenvalues of $T$ of finite algebraic multiplicity) is 
abbreviated by $\sigma_d(T)$. 

The Banach spaces of bounded
and compact linear operators in $\cH$ are denoted by $\cB(\cH)$ and
$\cB_\infty(\cH)$, respectively. Similarly, the Schatten--von Neumann
(trace) ideals will subsequently be denoted by $\cB_p(\cH)$,
$p \in [1,\infty)$, and the subspace of all finite rank operators in $\cB_1(\cH)$ will be 
abbreviated by $\cF(\cH)$. Analogous notation $\cB(\cH_1,\cH_2)$,
$\cB_\infty (\cH_1,\cH_2)$, etc., will be used for bounded, compact, etc.,
operators between two Hilbert spaces $\cH_1$ and $\cH_2$. In addition,
$\tr_{\cH}(T)$ denotes the trace of a trace class operator $T\in\cB_1(\cH)$ and
$\det_{p,\cH}(I_{\cH}+S)$ represents the (modified) Fredholm determinant
associated with an operator $S\in\cB_p(\cH)$, $p\in\bbN$ (for $p=1$ we
omit the subscript $1$). 
In addition, $\Phi(\cH)$ denotes the set of bounded Fredholm operators on $\cH$ (i.e., 
the set of operators $T \in \cB(\cH)$ such that $\dim(\ker(T)) < \infty$, $\ran(T)$ is 
closed in $\cH$, and $\dim(\ker(T^*)) < \infty$). The corresponding (Fredholm) index 
of $T \in \Phi(\cH)$ is then given by $\ind(T) = \dim(\ker(T)) - \dim(\ker(T^*))$. 

The symbol $\dotplus$ denotes a direct (but not necessary orthogonal direct) decomposition 
in connection with subspaces of Banach spaces.

Finally, we denote by $D(z_0; r_0) \subset \bbC$ the open disk with center $z_0$ and radius 
$r_0 > 0$, and by $C(z_0; r_0) = \partial D(z_0; r_0)$ the corresponding circle.

\section{Abstract Perturbation Theory} \lb{s2}

In this introductory section, following Kato \cite{Ka66}, Konno and Kuroda
\cite{KK66}, and Howland \cite{Ho70}, we consider a class
of factorable non-self-adjoint perturbations of a given unperturbed
non-self-adjoint operator. We closely follow the treatment in \cite{GLMZ05} 
(in which $H_0$ is explicitly permitted to be non-self-adjoint, cf.\ Hypothesis \ref{h2.1}\,$(i)$ 
below) and refer to the latter for detailed proofs.

We start with our first set of hypotheses.

\begin{hypothesis} \lb{h2.1}
$(i)$ Suppose that $H_0\colon\dom(H_0)\to\cH$,
$\dom(H_0)\subseteq\cH$ is a densely defined, closed, linear operator
in $\cH$ with nonempty resolvent set,
\begin{equation}
\rho(H_0)\neq\emptyset, \lb{2.1}
\end{equation}
$V_1 \colon \dom(V_1)\to\cK$, $\dom(V_1)\subseteq\cH$ a densely defined,
closed, linear operator from $\cH$ to $\cK$, and
$V_2 \colon \dom(V_2)\to\cK$, $\dom(V_2)\subseteq\cH$ a densely
defined, closed, linear operator from $\cH$ to $\cK$ such that
\begin{equation}
\dom(V_1)\supseteq\dom(H_0), \quad \dom(V_2)\supseteq\dom(H_0^*).
\lb{2.2}
\end{equation}
In the following we denote
\begin{equation}
R_0(z)=(H_0-zI_{\cH})^{-1}, \quad z\in \rho(H_0). \lb{2.3}
\end{equation}
$(ii)$ For some (and hence for all) $z\in\rho(H_0)$, the operator
$- V_1R_0(z)V_2^*$, defined on $\dom(V_2^*)$, has a
bounded extension in $\cK$, denoted by $K(z)$,
\begin{equation}
K(z)=-\ol{V_1 R_0(z) V_2^*} \in\cB(\cK). \lb{2.4}
\end{equation}
$(iii)$ $1\in\rho(K(z_0))$ for some $z_0\in \rho(H_0)$.
\end{hypothesis}

That $K(z_0)\in\cB(\cK)$ for some $z_0\in\rho(H_0)$ implies
$K(z)\in\cB(\cK)$ for all $z\in\rho(H_0)$ (as mentioned in
Hypothesis \ref{h2.1}\,(ii)) is an immediate consequence of
\eqref{2.2} and the resolvent equation for $H_0$.

We emphasize that in the case where $H_0$ is self-adjoint, the
following results in Lemma \ref{l2.1}, Theorem \ref{t2.3}, and
Remark \ref{r2.4} are due to Kato \cite{Ka66} (see also \cite{Ho70},
\cite{KK66}). The more general case we consider here requires only
minor modifications, but for the convenience of the reader we will
sketch most of the proofs.

\begin{lemma} \lb{l2.1}
Let $z,z_1,z_2\in\rho(H_0)$. Then  Hypothesis \ref{h2.1} implies the
following facts:
\begin{align}
& V_1 R_0(z)\in\cB(\cH,\cK), \quad \ol{R_0(z) V_2^*}
= \big[V_2 (H_0^* - {\ol z} I_{\cH})^{-1}\big]^*\in
\cB(\cK,\cH), \lb{2.5}
\\ & \ol{R_0(z_1) V_2^*}-\ol{R_0(z_2) V_2^*}
=(z_1-z_2)R_0(z_1)\ol{R_0(z_2) V_2^*} \lb{2.6} \\
& \hspace*{3.35cm} =(z_1-z_2)R_0(z_2)\ol{R_0(z_1) V_2^*}, \lb{2.7} \\
& K(z)=- V_1 \ol{[R_0(z) V_2^*]}, \quad K(\ol z)^*=- V_2\ol{[R_0(\ol z)^* V_1^*]},
\lb{2.8} \\
& \ran\big(\ol{R_0(z) V_2^*}\big)\subseteq \dom(V_1), \quad
\ran\big(\ol{R_0(\ol z)^* V_1^*}\big)\subseteq \dom(V_2), \lb{2.9} \\
& K(z_1)-K(z_2)=(z_2-z_1) V_1 R_0(z_1)\ol{R_0(z_2) V_2^*} \lb{2.10} \\
& \hspace*{2.25cm} =(z_2-z_1) V_1 R_0(z_2)\ol{R_0(z_1) V_2^*}. \lb{2.11}
\end{align}
\end{lemma}

Next, following Kato \cite{Ka66}, one introduces
\begin{align}
\begin{split}
& R(z)=R_0(z)-\ol{R_0(z) V_2^*}[I_{\cK}-K(z)]^{-1} V_1 R_0(z), \lb{2.13} \\
& \hspace*{2.65cm} z\in\{\zeta\in\rho(H_0)\,|\, 1\in\rho(K(\zeta))\}.
\end{split}
\end{align}

\begin{theorem} \lb{t2.3}
Assume Hypothesis \ref{h2.1} and suppose
$z\in\{\zeta\in\rho(H_0)\,|\, 1\in\rho(K(\zeta))\}$. Then, $R(z)$
defined in \eqref{2.13} defines a densely defined, closed, linear
operator $H$ in $\cH$ by
\begin{equation}
R(z)=(H-zI_{\cH})^{-1}. \lb{2.14}
\end{equation}
Moreover,
\begin{equation}
V_1 R(z),  V_2 R(z)^* \in \cB(\cH,\cK) \lb{2.15}
\end{equation}
and
\begin{align}
R(z)&=R_0(z)-\ol{R(z) V_2^*} V_1 R_0(z) \lb{2.16} \\
&=R_0(z)-\ol{R_0(z) V_2^*} V_1 R(z).   \lb{2.17}
\end{align}
Finally, $H$ is an extension of
$(H_0 + V_2^* V_1)|_{\dom(H_0)\cap\dom(V_2^* V_1)}$
$($the latter intersection domain may consist of $\{0\}$ only$)$,
\begin{equation}
H\supseteq (H_0 + V_2^* V_1)|_{\dom(H_0)\cap\dom(V_2^* V_1)}. \lb{2.18}
\end{equation}
\end{theorem}

\begin{remark} \lb{r2.4}
$(i)$ Assume that $H_0$ is self-adjoint in $\cH$. Then $H$ is also
self-adjoint if
\begin{equation}
(V_1 f, V_2g)_{\cK}=(V_2 f,V_1 g)_{\cK} \, \text{ for all } \,
f,g\in\dom(V_1) \cap \dom(V_2). \lb{2.24}
\end{equation}
$(ii)$ The formalism is symmetric with respect to $H_0$ and $H$ in the
following sense: The densely defined operator $- V_1 R(z) V_2^*$ has a
bounded extension to all of
$\cK$ for all $z\in\{\zeta\in\rho(H_0)\,|\, 1\in\rho(K(\zeta))\}$, in
particular,
\begin{equation}
I_{\cK}-\ol{V_1 R(z) V_2^*}=[I_{\cK}-K(z)]^{-1}, \quad
z\in\{\zeta\in\rho(H_0)\,|\, 1\in\rho(K(\zeta))\}.   \lb{2.25}
\end{equation}
Moreover,
\begin{align}
\begin{split}
R_0(z)&=R(z)+\ol{R(z) V_2^*}\big[I_{\cK}-\ol{V_1 R(z) V_2^*}\big]^{-1} V_1 R(z), \\
& \hspace*{2.35cm}
z\in\{\zeta\in\rho(H_0)\,|\, 1\in\rho(K(\zeta))\},    \lb{2.26}
\end{split}
\end{align}
and
\begin{equation}
H_0\supseteq (H - V_2^* V_1)|_{\dom(H) \cap \dom(V_2^* V_1)}.  \lb{2.27}
\end{equation}
$(iii)$ The basic hypotheses \eqref{2.2} which amount to
\begin{equation}
V_1 R_0(z)\in\cB(\cH,\cK), \quad \ol{R_0(z) V_2^*} 
= \big[V_2 (H_0^* - {\ol z} I_{\cH})^{-1}\big]^* \in
\cB(\cK,\cH), \quad z\in\rho(H_0).  \lb{2.27a}
\end{equation}
(cf.\ \eqref{2.5}) are more general than a quadratic form perturbation
approach which would result in conditions of the form
\begin{align}
\begin{split} 
V_1 R_0(z)^{1/2}\in\cB(\cH,\cK), \quad \ol{R_0(z)^{1/2} V_2^*}
= \big[V_2 (H_0^* - {\ol z} I_{\cH})^{-1/2}\big]^* \in \cB(\cK,\cH),& \\ 
z\in\rho(H_0),&  \lb{2.27b}
\end{split} 
\end{align}
or even an operator perturbation approach which would involve conditions
of the form
\begin{equation}
[V_2^* V_1] R_0(z)\in\cB(\cH), \quad z\in\rho(H_0).  \lb{2.27c}
\end{equation}
\end{remark}

The next result represents an abstract version
of (a variant of) the Birman--Schwinger principle due to Birman \cite{Bi61}
and Schwinger \cite{Sc61} (cf.\ also \cite{BS91}, \cite{GH87},
\cite{Kl82}, \cite{KS80}, \cite{Ne83}, \cite{Ra80}, \cite{Se74}, \cite[Ch.\ III]{Si71}, and
\cite{Si77a}). We will focus on geometric multiplicities and again 
follow \cite{GLMZ05} closely. 

We need to strengthen our hypotheses a bit and hence introduce the following
assumption:

\begin{hypothesis} \lb{h2.5}
In addition to Hypothesis \ref{h2.1} we suppose the condition: \\
$(iv)$ $[I_{\cK} - K(z)] \in \Phi(\cK)$ for all $z\in\rho(H_0)$.
\end{hypothesis}

\begin{remark} \lb{r2.6}
In concrete applications, say, to Schr\"odinger-type operators, condition $(iv)$ in 
Hypothesis \ref{h2.5} is frequently replaced by the stronger assumption: \\
$(iv')$ $K(z)\in\cB_\infty(\cK)$ for all $z\in\rho(H_0)$.   \\
In this case $[I_{\cK} - K(z)]$ is a Fredholm operator with index zero for all $z\in\rho(H_0)$.

An elementary example illustrating that condition $(iv')$ is stronger than $(iv)$ is easily 
constructed as follows: Choose 
\begin{align}
\begin{split} 
& \cK = \cH, \quad V_j = I_{\cH}, \; j=1,2, \quad H_0 = H_0^*, \quad H = H_0 + I_{\cH},   \\
& \sigma_p(H_0) = \emptyset, \quad \sigma_{ess} (H_0) \neq \emptyset.  
\end{split}
\end{align} 
Next, choose $z \in \bbC\backslash \bbR$. Then,  $K(z) = (H_0 - z I_{\cH})^{-1}$, and 
$I_{\cH} - (H_0 - z I_{\cH})^{-1} = (H_0 - (z + 1) I_{\cH})(H_0 - z I_{\cH})^{-1}$, implying that 
$\ran(I_{\cH} - K(z)) = \cH$ and $\ker(I_{\cH} - K(z)) = \ker(I_{\cH} - K(z)^*) = \{0\}$. Hence, 
$[I_{\cK} - K(z)] \in \Phi(\cK)$ with $\ind(I_{\cK} - K(z)) =0$. However, since 
$\sigma_{ess}(K(z)) \supsetneqq \{0\}$, one concludes that $K(z) \notin \cB_{\infty}(\cH)$. 
\end{remark}

Since by \eqref{2.25},
\begin{equation}
-\ol{V_1 R(z) V_2^*} = [I_{\cK}-K(z)]^{-1}K(z) =-I_{\cK}+[I_{\cK}-K(z)]^{-1},       \lb{2.28}
\end{equation}
Hypothesis \ref{h2.5} implies that $I_{\cK} - \ol{V_1 R(\cdot) V_2^*}$ extends to a 
Fredholm operator in $\Phi(\cK)$ as long as the right-hand side of \eqref{2.28} exists.  
Similarly, under condition $(iv')$ in Remark \ref{r2.6}, $\ol{V_1 R(\cdot) V_2^*}$ extends to a compact 
operator in $\cK$ as long as the right-hand side of \eqref{2.28} exists.

Regarding eigenvalues, we recall that if $T$ is a densely defined, closed linear operator in $\cH$, 
then the {\it geometric multiplicity}, $m_g(z_0;T)$, of an eigenvalue $z_0 \in \sigma_p(T)$ of $T$ 
is given by 
\begin{equation}
m_g(z_0;T) = \dim(\ker(T - z_0 I_{\cH})).
\end{equation}  

The following general result is due to Konno and Kuroda \cite{KK66} in the case where $H_0$ 
is self-adjoint and condition $(iv')$ in Remark \ref{r2.6} is assumed. 

\begin{theorem}[\cite{KK66}] \lb{t2.7}
Assume Hypothesis \ref{h2.5} and let $z_0\in\rho(H_0)$. Then,
\begin{equation}
Hf=z_0 f, \quad 0\neq f\in\dom(H) \, \text{ implies } \,
K(z_0)g=g \lb{2.29}
\end{equation}
where, for fixed $z_1\in\{\zeta\in\rho(H_0)\,|\,
1\in\rho(K(\zeta))\}$, $z_1\neq z_0$,
\begin{align}
0 \neq g &= [I_{\cK}-K(z_1)]^{-1} V_1 R_0(z_1)f \lb{2.30} \\
&= (z_0-z_1)^{-1} V_1 f. \lb{2.30a}
\end{align}
Conversely,
\begin{equation}
K(z_0)g=g, \quad 0\neq g\in\cK \, \text{ implies } \,
Hf=z_0 f, \lb{2.31}
\end{equation}
where
\begin{equation}
0\neq f=-\ol{R_0(z_0) V_2^*}g\in\dom(H). \lb{2.32}
\end{equation}
Moreover,
\begin{equation}
m_g(z_0;H) = \dim(\ker(H-z_0I_{\cH}))
=\dim(\ker(I_{\cK}-K(z_0))) = m_g(1;K(z_0)) <\infty.
\lb{2.33}
\end{equation}
In particular, let $z\in\rho(H_0)$, then
\begin{equation}
\text{$z\in\rho(H)$ if and only if\;
$1\in\rho(K(z))$.} \lb{2.33a}
\end{equation}
\end{theorem}

It is possible to avoid the Fredholm operator (resp., compactness) assumption in condition $(iv)$ in 
Hypothesis \ref{h2.5} (resp., condition $(iv')$ in Remark \ref{r2.6}) in Theorem \ref{t2.7} 
provided that \eqref{2.33} is replaced by the statement:
\begin{equation}\label{2.33NEW}
\text{The subspaces $\ker(H-z_0I_{\cH})$ and
$\ker(I_{\cK}-K(z_0))$ are isomorphic}
\end{equation}
(cf.\ \cite{GLMZ05}). Of course, \eqref{2.33} follows from \eqref{2.33NEW} provided
$\ker(I_{\cK}-K(z_0))$ is finite-dimensional, which in turn follows from Hypothesis \ref{h2.5}.

\section{On Factorizations of Analytic Operator-Valued Functions}
\lb{s3}

In this section, we consider factorizations of analytic operator-valued functions.  We recall and extend a factorization result due to Howland \cite{Ho71}.

Assuming $\Omega \subseteq \bbC$  to be open and $M(\cdot)$ to be a $\cB(\cH)$-valued 
meromorphic function on $\Omega$ that has the norm convergent Laurent expansion around 
$z_0 \in \Omega$ of the type 
\begin{align}
\begin{split} 
M(z) = \sum_{k= - N_0}^\infty (z - z_0)^{k} M_k(z_0), \quad M_k(z_0) \in \cB(\cH), \, k \in \bbZ, 
\; k \geq - N_0,&      \\  
0 < |z - z_0| < \varepsilon_0,&      \lb{3.1}
\end{split} 
\end{align}
for some $N_0 = N_0(z_0) \in \bbN$ and some $0 < \varepsilon_0 = \varepsilon_0(z_0)$ sufficiently small, 
we denote the principal part, ${\rm pp}_{z_0} \, \{M(\cdot)\}$, of $M(\cdot)$ at $z_0$ by
\begin{align}
\begin{split} 
{\rm pp}_{z_0} \, \{M(z)\} = \sum_{k= - N_0}^{-1} (z - z_0)^{k} M_k(z_0), \quad 
M_{-k}(z_0) \in \cB(\cH), \; 1 \leq k \leq N_0,&     \\
0 < |z - z_0| < \varepsilon_0.&    \lb{3.2}
\end{split} 
\end{align}

Given the notation \eqref{3.2}, we start with the following definition.

\begin{definition} \lb{d3.1}
Let $\Omega\subseteq\bbC$ be open and connected. Suppose that $M(\cdot)$ is a 
$\cB(\cH)$-valued analytic function on $\Omega$ except for isolated singularities. Then  
$M(\cdot)$ is called {\it finitely meromorphic at $z_0 \in\Omega$} if $M(\cdot)$ is 
analytic on the punctured disk $D(z_0;\varepsilon_0)\backslash\{z_0\} \subset \Omega$ 
centered at $z_0$ with sufficiently small $\varepsilon_0>0$, and the principal part of $M(\cdot)$ 
at $z_0$ is of finite rank, that is, if the principal part of $M(\cdot)$ is of the type \eqref{3.2}, and
one has 
\begin{equation}
M_{-k}(z_0) \in \cF(\cH), \quad 1 \leq k \leq N_0. 
\end{equation}
In addition, $M(\cdot)$ is called {\it finitely meromorphic on $\Omega$} if it is meromorphic 
on $\Omega$ and finitely meromorphic at each of its poles.
\end{definition}

In using the term {\it finitely meromorphic} we closely follow the convention in \cite{GS71} (see 
also \cite[Sect.\ XI.9]{GGK90} and \cite[Sect.\ 4.1]{GL09}).  We also note that the notions 
{\it completely meromorphic} (cf.\ \cite{Ho70}, adopted in \cite{GLMZ05}) and {\it essentially meromorphic} (cf.\ \cite{RV69}) have been used instead in the literature. 

Throughout this section we make the following assumptions:

\begin{hypothesis} \lb{h3.2} 
Let $\Omega \subseteq \bbC$ be open and connected, and suppose that 
$A:\Omega \to \cB(\cH)$ is analytic and that 
\begin{equation}
A(z) \in \Phi(\cH) \, \text{ for all } \, z\in\Omega.    \lb{3.4} 
\end{equation}  
\end{hypothesis}

One then recalls the analytic Fredholm theorem in the following form:

\begin{theorem} [\cite{GL09}, Sect.~4.1, \cite{GS71}, \cite{Ho70}, {\cite[Theorem\ VI.14]{RS80}}, \cite{St68}] \lb{t3.3}  ${}$ \\
Assume that $A:\Omega \to \cB(\cH)$ satisfies Hypothesis \ref{h3.2}. Then either \\
$(i)$ $A(z)$ is not boundedly invertible for any $z \in \Omega$, \\[1mm]
or else, \\[1mm]
$(ii)$ $A(\cdot)^{-1}$ is finitely meromorphic on $\Omega$. More precisely, there exists a discrete 
subset $\cD_1 \subset \Omega$ $($possibly, $\cD_1 = \emptyset$$)$ such that 
$A(z)^{-1} \in \cB(\cH)$ for all $z \in \Omega\backslash\cD_1$, $A(\cdot)^{-1}$ is analytic on 
$\Omega\backslash\cD_1$, and meromorphic on $\Omega$. In addition, 
\begin{equation}
A(z)^{-1} \in \Phi(\cH) \, \text{ for all } \, z \in \Omega\backslash\cD_1, 
\end{equation}
and if $z_1 \in \cD_1$ then 
\begin{equation}
A(z)^{-1} = \sum_{k= - N_0(z_1)}^{\infty} (z - z_1)^{k} C_k(z_1), \quad 
0 < |z - z_1| < \varepsilon_0(z_1),    \lb{3.5}
\end{equation}
with  
\begin{align}\lb{3.7B} 
\begin{split} 
& C_{-k}(z_1) \in \cF(\cH), \; 1 \leq k \leq N_0(z_1),  \quad  C_0(z_1) \in \Phi(\cH), \\
& C_k(z_1) \in \cB(\cH), \; k \in \bbN. 
\end{split} 
\end{align}
In addition, if $[I_{\cH} - A(z)] \in \cB_{\infty}(\cH)$ for all $z \in \Omega$, then 
\begin{equation}
\big[I_{\cH} - A(z)^{-1}\big] \in \cB_{\infty}(\cH), \; z \in \Omega\backslash\cD_1, 
\quad [I_{\cH} - C_0(z_1)] \in \cB_{\infty}(\cH), \; z_1 \in \cD_1.  
\end{equation} 
\end{theorem}

For an interesting extension of the analytic Fredholm theorem in connection with Hahn 
holomorphic functions we refer to \cite{MS14}.

For a linear operator $S$ in $\cH$ with closed range one defines the 
{\it defect of $S$}, denoted by $\de(S)$, by the codimension of $\ran(S)$ in 
$\cH$, that is, 
\begin{equation}
\de(S) = \dim\big(\ran(S)^{\bot}\big). 
\end{equation}
 
In addition, we recall the notion of linear independence with respect to a linear subspace of $\cH$: Let $\cD \subseteq \cH$ be a linear subspace of $\cH$. 
Then vectors $f_k \in \cH$, $1 \leq k \leq N$, $N \in \bbN$, 
are called {\it linearly independent} $({\rm mod}\, \cD)$, if 
\begin{align}
\begin{split} 
&\sum_{k = 1}^N c_k f_k \in \cD \, \text{ for some coefficients } 
\, c_k \in \bbC, \, 1 \leq k \leq N, \\
& \quad \text{implies } \, c_k = 0, \, 1 \leq k \leq N.     \lb{3.10}
\end{split} 
\end{align}
In addition, with $\cD$ and $\cE$ linear subspaces of $\cH$ with 
$\cD \subseteq \cE$, the quotient subspace 
$\cE / \cD$ consists of equivalence classes $[f]$ such that $g \in [f]$ if and 
only if $(f - g) \in \cD$, in particular, $f = g \; ({\rm mod}\, \cD)$ is equivalent 
to $(f -g) \in \cD$. Moreover, the dimension of 
$\cE \; ({\rm mod} \, \cD)$, denoted by $\dim_{\cD}(\cE)$, equals 
$n \in \bbN$, if there are $n$, but not more than $n$, linearly 
independent vectors in $\cE$, such that no linear combination (except, the 
trivial one) belongs to $\cD$. If no such finite $n\in\bbN$ exists, one defines 
$\dim_{\cD}(\cE) = \infty$. 
 
The following three results due to Howland \cite{Ho71} are fundamental for the remainder 
of this section and for convenience of the reader we include their proof under slightly more 
general hypotheses, replacing Howland's assumption that 
$[A(\cdot) - I_{\cH}] \in \cB_{\infty}(\cH)$ by the assumption that $A(\cdot)$ is Fredholm. 
n addition, we occasionally offer a few additional details in the proofs of these results.

\begin{theorem} [\cite{Ho71}] \lb{t3.4}
Assume that $A:\Omega \to \cB(\cH)$ satisfies Hypothesis \ref{h3.2}, suppose that $A(z)$ is 
boundedly invertible for some $z \in \Omega$ $($i.e., case $(ii)$ in Theorem \ref{t3.3} applies$)$, 
and let $z_0 \in \Omega$ be a pole of $A(\cdot)^{-1}$. Denote by $Q_1$ any projection onto 
$\ran(A(z_0))$ and let $P_1 = I_{\cH} - Q_1$. Then,
\begin{equation}
A(z) = [Q_1 - (z-z_0) P_1] A_1(z), \quad z \in \Omega,     \lb{3.10a} 
\end{equation}
where 
\begin{align} 
& \text{$A_1(\cdot)$ is analytic on $\Omega$,} \lb{3.12o}\\
& A_1(z) \in \Phi(\cH), \quad z\in\Omega,    \lb{3.11} \\ 
& \ind(A(z)) = \ind(A_1(z)) =0, \quad z \in \Omega, \; |z - z_0| \, \text{ sufficiently small,}   
\lb{3.13} \\ 
& \de(A_1(z_0)) \leq \de(A(z_0)).    \lb{3.12} 
\end{align}
If $z_0$ is a pole of $A(\cdot)^{-1}$ of order $n_0\in\bbN$, then $z_0$ is a pole 
of $A_1(\cdot)^{-1}$ of order $n_0 - 1$. Finally, 
\begin{align}
& [I_{\cH} - A(\cdot)] \in \cF(\cH) \, \text{ $($resp., $\cB_p(\cH)$ for some $1 \leq p \leq \infty$$)$}  \lb{3.14} \\ 
& \quad \text{if and only if } \, 
[I_{\cH} - A_1(\cdot)] \in \cF(\cH) \, \text{ $($resp., $\cB_p(\cH)$ for some $1 \leq p \leq \infty$$)$.}   \no 
\end{align}
\end{theorem}
\begin{proof} 
In the following let $z \in \Omega$. Since by hypothesis $z_0$ is a pole of $A(\cdot)^{-1}$ and hence an isolated singularity of 
$A(\cdot)^{-1}$, the second alternative of the analytic Fredholm theorem, Theorem \ref{t3.3}, 
is realized.  Due to assumption \eqref{3.4}, $\ran(A(z_0))$ 
 is closed in $\cH$. With respect to the decomposition $\cH = P_1\cH \dotplus Q_1\cH$ 
 one infers that 
\begin{equation}\lb{3.16B}
Q_1 - (z-z_0) P_1 = \begin{pmatrix} - (z-z_0) P_1 & 0 \\ 0 & Q_1 \end{pmatrix}.
\end{equation} 
In addition, the projection $P_1$ is finite-dimensional which may be seen as follows.  By \cite[p.~156, 267]{Ka80}, the adjoint $P_1^*$ is a projection onto $\ran(A(z_0))^{\perp}=\ker(A(z_0)^*)$.  However, $A(z_0)^*\in \Phi(\cH)$ since $A(z_0)\in \Phi(\cH)$, so that $P_1^*$ is a finite-dimensional projection,
\begin{equation}\lb{3.17B}
\dim(\ran((P_1^*)) = \dim(\ker(A(z_0)^*))<\infty.
\end{equation}
Evidently, \eqref{3.17B} also implies (cf., e.g., \cite[Theorem 6.1]{We80})
\begin{equation}
\dim(\ran((P_1)) = \dim(\ker(A(z_0)^*)) = \dim(\ran(P_1^*)) < \infty.    \lb{3.19AA}
\end{equation}
Next, the representation in \eqref{3.16B} implies
\begin{align}
& [Q_1 - (z-z_0) P_1]^{-1} 
= \begin{pmatrix} - (z-z_0)^{-1} P_1 & 0 \\ 0 & Q_1 \end{pmatrix}    \no \\
& \hspace*{3.02cm} = Q_1 - (z-z_0)^{-1} P_1, \quad z \in \Omega \backslash \{z_0\},    \lb{3.14a} \\ 
& {\det}_{\cH} (Q_1 - (z-z_0) P_1) = (z_0 - z)^{p_1}, \quad p_1 = \dim(\ran(P_1)).   
\lb{3.15} 
\end{align}
Thus,
\begin{equation}
A_1(z) = [Q_1 - (z-z_0) P_1]^{-1} A(z) = Q_1 A(z) - (z - z_0)^{-1} P_1 [A(z) - A(z_0)] 
\lb{3.22A} 
\end{equation}
is analytic on $\Omega$ since $P_1 A(z_0) = P_1 Q_1 A(z_0) = 0$ as $Q_1$ acts on 
$\ran(A(z_0))$ as the identity operator by hypothesis. In particular, 
\begin{equation}
A_1(z_0) = Q_1 A(z_0) - P_1 A'(z_0).     \lb{3.23A}
\end{equation}

Moreover, using once more that 
\begin{equation}
A_1(z) = \big[Q_1 - (z-z_0)^{-1} P_1\big] A(z), \quad z \in \Omega \backslash \{z_0\},   
\lb{3.18} 
\end{equation}
one notices that by hypothesis, $A(\cdot) \in \Phi(\cH)$ on $\Omega$, and that by 
\eqref{3.14a},
\begin{equation}
\big[Q_1 - (z-z_0)^{-1} P_1\big]^{-1} = [Q_1 - (z-z_0) P_1] \in \cB(\cH),\quad z \in \Omega \backslash \{z_0\},
\end{equation}
and analogously for its adjoint. In particular,  
$\ran\big([Q_1 - (z-z_0) P_1]^{-1}\big) = \cH$, and hence one also concludes that 
$\big[Q_1 - (z-z_0)^{-1} P_1\big] \in \Phi(\cH)$ for $z \in \Omega \backslash \{z_0\}$, and hence $A_1(z) \in \Phi(\cH)$ for $z \in \Omega \backslash \{z_0\}$. The remaining case $z=z_0$ now follows from \eqref{3.23A}, since $P_1$ and hence $P_1 A'(z_0)$ is of finite rank (thus, compact), $Q_1 = [I_{\cH} - P_1] \in \Phi(\cH)$ and $A(z_0) \in \Phi(\cH)$, implying $Q_1 A(z_0) \in \Phi(\cH)$, and the fact that a Fredholm operator plus a compact operator is again Fredholm (cf., e.g., \cite[Theorem~5.10]{Sc02}). 

Invariance of the Fredholm index as recorded in \eqref{3.13} 
is shown as follows. First, by Hypothesis \ref{h3.2} and alternative $(ii)$ in Theorem \ref{t3.3}, 
$\ind(A(z)) =0$ whenever $A(z)$ is boundedly invertible, in particular, this holds for a 
sufficiently small, punctured disk $D(z_0; \varepsilon) \backslash \{z_0\} \subset \Omega$ with center $z_0$, that is, for $0 < |z - z_0| < \varepsilon_0$ for $0 < \varepsilon_0$ sufficiently small. Since $A(\cdot)$ is analytic in $\Omega$, a perturbation argument, writing 
$A(z_0) = A(z) + [A(z_0) - A(z)]$ then yields (cf.\ \cite[Theorem~5.11]{Sc02})
\begin{equation}
\ind(A(z_0)) = \ind(A(z)) = 0,
\end{equation}
choosing $0 < \varepsilon_0$ sufficiently small. Precisely the same arguments apply to 
$\ind(A_1(\cdot))$ and hence yield \eqref{3.13}.  
 
Next, since $A(z_0) = Q_1 A_1(z_0)$, $Q_1$ maps $\ran(A_1(z_0))$ onto 
$\ran(A(z_0))$. Consequently, if $Q_1^* f \in \ran(A_1(z_0))^{\bot}$ for some $f \in \cH$, 
then 
\begin{equation} 
(f, A(z_0)g)_{\cH} = (f, Q_1 A_1(z_0)g)_{\cH} = (Q_1^*f, A_1(z_0)g)_{\cH} = 0 \, 
\text{ for all $g \in \cH$} 
\end{equation} 
implies $f \in \ran(A(z_0))^{\bot}$. Hence, one concludes that $Q_1^* f =0$ since 
$\ker(Q_1^*) = \ran(Q_1)^{\bot} = \ran(A(z_0))^{\bot}$ and thus,  
\begin{equation}
\ran(Q_1^*) \cap \ran(A_1(z_0))^{\bot} = \{0\}   \lb{3.18A}
\end{equation} 
yields the existence of a finite-dimensional (hence, closed) linear subspace 
$\cH_1 \subset \cH$ such that 
\begin{equation}
\cH = \ran(Q_1^*) \dotplus \ran(A_1(z_0))^{\bot} \dotplus \cH_1.   \lb{3.18B}
\end{equation}
Using the fact that 
\begin{equation}
\cH = \ran(Q_1^*) \dotplus \ker(Q_1^*) =
\ran(Q_1^*) \dotplus \ran(Q_1)^{\bot} = \ran(Q_1^*) \dotplus \ran(A(z_0))^{\bot}, 
\lb{3.18C}
\end{equation} 
one infers from \eqref{3.18B}, \eqref{3.18C}, and the paragraph following 
\eqref{3.10}, that 
\begin{align} \lb{3.18a} 
\begin{split} 
& \de(A_1(z_0)) = \dim\big(\ran(A_1(z_0))^{\bot}\big) 
= \dim_{[\ran(Q_1^*) \dotplus \cH_1]^{\perp}} (\cH)     \\
& \quad \leq \dim_{[\ran(Q_1^*)]^{\perp}} (\cH)  
= \dim\big(\ran(A(z_0))^{\bot}\big) = \de(A(z_0)).
\end{split} 
\end{align}
Finally, suppose that $A(\cdot)^{-1}$ has a pole of order $n_0 \in\bbN$ at $z_0$. 
Since $A(z)^{-1} = A_1(z)^{-1} [Q_1 - (z - z_0)^{-1} P_1 ]$, $A_1(z)^{-1}$ must 
have a pole at $z_0$ of order at least $n_0 -1$. Since 
\begin{equation}
A_1(z)^{-1} = - (z - z_0) A(z)^{-1} P_1 + A(z)^{-1} Q_1,   \lb{3.19a}
\end{equation}   
the order of the pole of the first term on the right-hand side of \eqref{3.19a} cannot exceed $n_0 - 1$. If 
$Q_1 f = A(z_0) g$, for some $f, g \in \cH$, one obtains 
\begin{align}
& (z_0 - z)^{n_0 -1} A(z)^{-1} Q_1 f = (z_0 - z)^{n_0 -1} A(z)^{-1} A(z_0) g   \no \\
& \quad = (z_0 - z)^{n_0 -1}  g + (z_0 - z)^{n_0} A(z)^{-1} (z - z_0)^{-1} [A(z) g - A(z_0) g]. 
\end{align}
For each fixed $f \in \cH$, the latter expression is uniformly bounded with respect to $z$ near 
$z_0$, and hence the uniform boundedness principle guarantees that also the pole of 
$A(z)^{-1} Q_1$ at $z_0$ cannot exceed $n_0 -1$. Thus, $A_1(\cdot)^{-1}$ has precisely a pole 
of order $n_0 - 1$ at $z_0$.  
To prove \eqref{3.14} for $z\neq z_0$, it suffices to note the pair of formulas
\begin{align}
I_{\cH} - A(z) &= [1 + (z - z_0)] P_1 + [Q_1 - (z - z_0) P_1] [I_{\cH} - A_1(z)],  \quad z \in \Omega, 
 \lb{3.19b} \\
 I_{\cH} - A_1(z) &= \big[1 + (z - z_0)^{-1}\big] P_1 + \big[Q_1 - (z - z_0)^{-1} P_1\big] [I_{\cH} - A(z)], 
\quad  z \in \Omega \backslash \{z_0\},       \lb{3.19c} 
\end{align}
and use the following facts:  $P_1\in \cF(\cH)\subset \cB_p(\cH)$, both $\cF(\cH)$ and $\cB_p(\cH)$ are closed under addition, and $\cB_p(\cH)$ is a two-sided ideal of $\cB(\cH)$, $1\leq p\leq \infty$.  To settle the case $z=z_0$, one uses \eqref{3.19b} to conclude that $[I_{\cH}-A(z_0)] \in \cF(\cH)$ (resp., $[I_{\cH}-A(z_0)] \in \cB_p(\cH))$ if 
$[I_{\cH}-A_1(z_0)] \in \cF(\cH)$ (resp., $[I_{\cH}-A_1(z_0)] \in \cB_p(\cH))$.  To arrive at the converse, one applies \eqref{3.10a} with $z=z_0$ to obtain
\begin{equation}
I_{\cH}-A_1(z_0) = I - (P_1+Q_1)A_1(z_0) = -P_1A_1(z_0) + [I_{\cH}-A(z_0)],
\end{equation}
noting that $P_1A_1(z_0)\in \cF(\cH)$.
\end{proof}

Still assuming that $A: \Omega \to \cB(\cH)$ satisfies Hypothesis \ref{h3.2} and that 
$A(\cdot)^{-1}$ has a pole at $z_0 \in \Omega$, we now decompose $\cH$ as follows. 
Introducing the Riesz projection $P(z)$ associated with $A(z)$, $z \in \cN(z_0)$ (cf., e.g., \cite[Sect.\
III.6]{Ka80}), with $\cN(z_0) \subset \Omega$ a sufficiently small neighborhood of $z_0$  
\begin{equation}
P(z)=\f{-1}{2\pi i} \ointctrclockwise_{\cC(0;\varepsilon_0)} d\zeta \, (A(z) -
\zeta I_{\cH})^{-1}, \quad z \in \cN(z_0),   \lb{3.21a}
\end{equation}
then $P(\dott)$ is analytic on $\cN(z_0)$ and 
\begin{equation}
\dim(\ran(P(z))) < \infty, \quad z \in \cN(z_0). 
\end{equation}
In addition, introduce the projections 
\begin{equation}
Q(z)=I_{\cH} - P(z), \quad z \in \cN(z_0).   \lb{3.22a}
\end{equation}
Next, following Wolf \cite{Wo52} one introduces the transformation
\begin{equation}
T(z)=P(z_0) P(z)+Q(z_0) Q(z) = P(z_0) P(z)+[I_{\cH}-P(z_0)][I_{\cH}-P(z)], \quad z \in \cN(z_0),
\lb{3.16}
\end{equation}
such that
\begin{equation}
P(z_0) T(z) = T(z) P(z), \quad Q(z_0) T(z) = T(z) Q(z), \quad z \in \cN(z_0).   \lb{3.17}
\end{equation}
In addition, for $|z-z_0|$ sufficiently small, also $T(\cdot)^{-1}$ is analytic, 
\begin{equation}
T(z)= I_{\cH} + \Oh(z-z_0), \quad |z - z_0| \, \text{ sufficiently small},      \lb{3.19A}
\end{equation}
and without loss of generality we may assume in the following that $T(\cdot)$ and $T(\cdot)^{-1}$ 
are analytic on $\cN(z_0)$. This yields the decomposition of $\cH$ into
\begin{equation}
\cH = P(z_0) \cH \dotplus Q(z_0) \cH     \lb{3.19B} 
\end{equation}
and the associated $2 \times 2$ block operator decomposition of $T(z) A(z) T(z)^{-1}$ into 
\begin{equation}
T(z) A(z) T(z)^{-1} = \begin{pmatrix} F(z) & 0 \\ 0 & G(z) \end{pmatrix}, \quad z \in \cN(z_0), 
\lb{3.19}
\end{equation}
where $F(\cdot)$ and $G(\cdot)$ are analytic on $\cN(z_0)$, and, again without loss of generality, 
$G(\cdot)$ is boundedly invertible on $\cN(z_0)$,  
\begin{equation}
G(z)^{-1} \in\cB(\ran(Q(z_0))),  \quad z \in \cN(z_0).     \lb{3.20} 
\end{equation}   

Next, we introduce more notation: Let $\Omega_0 \subseteq \bbC$ be open and 
connected and $f\colon\Omega_0\to\bbC\cup\{\infty\}$ be meromorphic and not
identically vanishing on $\Omega_0$. The multiplicity function $m(z;f)$, $z\in\Omega_0$, is then 
defined by
\begin{align}
m(z;f)&=\begin{cases} k, & \text{if $z$ is a zero of $f$ of order $k$,} \\
-k, & \text{if $z$ is a pole of order $k$,} \\
0, & \text{otherwise} \end{cases} \lb{3.21} \\
&= \f{1}{2\pi i}\ointctrclockwise_{C(z; \varepsilon) } d\zeta \,
\f{f'(\zeta)}{f(\zeta)}, \quad z\in\Omega_0,      \lb{3.22}
\intertext{for $\varepsilon>0$
sufficiently small. Here the circle $C(z; \varepsilon)$ is chosen sufficiently small such
that $C(z; \varepsilon)$ contains no other singularities or zeros of $f$ except, possibly, $z$. 
If $f$ vanishes identically on $\Omega_0$, one
defines} m(z;f)&= \infty,  \quad z\in\Omega_0. \lb{3.23}
\end{align}

Given the block decomposition \eqref{3.19}, we follow Howland in introducing the quantity 
$\nu(z_0; A(\cdot))$ by 
\begin{equation}
\nu(z_0; A(\cdot)) = \begin{cases} m(z_0; {\det}_{\ran(P(z_0))} (F(\cdot))), & 
\text{if ${\det}_{\ran(P(z_0))} (F(\cdot)) \not\equiv 0$ on $\cN(z_0)$},  \\
\infty, & \text{if ${\det}_{\ran(P(z_0))} (F(\cdot)) \equiv 0$ on $\cN(z_0)$.}
\end{cases}
\end{equation}
We also recall the abbreviation
\begin{equation}
m_g(0; A(z_0)) = \dim(\ker(A(z_0))).  
\end{equation}

Repeated applications of Theorem \ref{t3.4} then yields the following principal factorization result of \cite{Ho71} (again, we extend it to the case of Fredholm 
operators $A(\cdot)$):

\begin{theorem} [\cite{Ho71}] \lb{t3.5}
Assume that $A:\Omega \to \cB(\cH)$ satisfies Hypothesis \ref{h3.2} and let $z_0 \in \Omega$ 
be a pole of $A(\cdot)^{-1}$ of 
order $n_0 \in \bbN$. Then there exist projections $P_j$ and $Q_j = I_{\cH} - P_j$ in $\cH$ such that with 
$p_j = \dim(\ran(P_j))$, $1 \leq j \leq n_0$, one infers that 
\begin{equation}
A(z) = [Q_1 - (z - z_0) P_1] [Q_2 - (z - z_0) P_2] \cdots [Q_{n_0} - (z - z_0) P_{n_0}] A_{n_0}(z), 
\quad z \in \Omega,    \lb{3.26}
\end{equation}
and 
\begin{equation}
1 \leq p_{n_0} \leq p_{n_0 -1} \leq \cdots \leq p_2 \leq p_1 < \infty,     \lb{3.27}
\end{equation}
where
\begin{align} 
& \text{$A_{n_0}(\cdot)$ is analytic on $\Omega$,}    \lb{3.28} \\
& A_{n_0}(z) \in \Phi(\cH), \quad z\in\Omega,      \lb{3.29} \\ 
& \ind(A(z)) = \ind(A_{n_0}(z)) = 0, \quad z \in \Omega, \; |z - z_0| \, \text{ sufficiently small,}     \lb{3.29a} \\
& A_{n_0}(z)^{-1} \in \cB(\cH), \quad z \in \Omega, \; |z - z_0| \, \text{ sufficiently small.}      \lb{3.30} 
\end{align}
In addition, 
\begin{equation}
p_1 = \dim(\ker(A(z_0)) = m_g(0; A(z_0)),    \lb{3.31}
\end{equation}
and hence
\begin{equation}
\nu(z_0; A(\cdot)) = \sum_{j=1}^{n_0} p_j \geq m_g(0; A(z_0)), 
\quad \nu(z_0; A(\cdot)) \geq n_0.   \lb{3.32}
\end{equation}
Finally, 
\begin{align}
& [I_{\cH} - A(\cdot)] \in \cF(\cH) \, \text{ $($resp., $\cB_p(\cH)$ for some $1 \leq p \leq \infty$$)$}  \lb{3.42} \\ 
& \quad \text{if and only if } \, 
[I_{\cH} - A_{n_0}(\cdot)] \in \cF(\cH) \, \text{ $($resp., $\cB_p(\cH)$ for some $1 \leq p \leq \infty$$)$.}  \no 
\end{align}
\end{theorem}
\begin{proof} 
Since by hypothesis $z_0$ is a pole of $A(\cdot)^{-1}$ and hence an isolated singularity of 
$A(\cdot)^{-1}$, the second alternative of the analytic Fredholm theorem, Theorem \ref{t3.3}, 
is realized. Applying Theorem \ref{t3.4} $n_0$ times, one obtains the $n_0$ factors  
$[Q_j - (z - z_0) P_j]$ and ends up with the facts \eqref{3.28}--\eqref{3.29a}.  In addition, $z_0$ is \textit{not} a pole of $A_{n_0}(\,\cdot\,)^{-1}$.

The bounded invertibility of $A_{n_0}(\cdot)$ in a suffficiently small punctured disk centered at 
$z_0$ is clear from that of $A(\cdot)$, \eqref{3.14a}, and \eqref{3.26}. To prove that also 
$A_{n_0}(z_0)^{-1} \in \cB(\cH)$ one can argue as follows. Since $A_{n_0}(\cdot)^{-1}$ has no pole 
at $z=z_0$, $A_{n_0}(z_0)$ is injective, $\ker(A_{n_0}(z_0)) = \{0\}$. In addition, since $A_{n_0}(z_0)$ is bounded and hence closed, $A_{n_0}(z_0)^{-1}$ is closed as well (cf., e.g., \cite[p.\ 165]{Ka80}). As $\ind(A_{n_0}(z_0)) = \dim(\ker(A_{n_0}(z_0))) = 0$, also $\dim(\ker(A_{n_0}(z_0)^*)) = 0$, and hence 
$\ran(A_{n_0}(z_0)) = \cH$, since $\ran(A_{n_0}(z_0))$ is closed in $\cH$. Hence, $A_{n_0}(z_0)^{-1}$ is 
defined on all of $\cH$ and an application of the closed graph theorem (see, e.g., 
\cite[Theorem\ III.5.20]{Ka80}) then yields $A_{n_0}(z_0)^{-1} \in \cB(\cH)$ and hence proves \eqref{3.30}.  

Equation \eqref{3.31} is clear from the decomposition 
$\cH = P_1 \cH \dotplus Q_1 \cH$ with $Q_1 \cH = \ran(A(z_0))$.  In light of the second alternative of the analytic Fredholm theorem, in particular \eqref{3.5} and \eqref{3.7B}, $A(\,\cdot\,)^{-1}$ has a finite-dimensional residue at $z_0$. In addition, 
\begin{equation} 
\dim(\ker(A(z_0)^{\sharp})) = \dim(\ran(P_1^{\sharp})) = m_g(z_0; A(z_0)^{\sharp}) 
= p_1 < \infty,
\end{equation}   
where $T^{\sharp}$ represents $T$ or $T^*$. 
The inequality \eqref{3.12} for the defects then yields the inequalities \eqref{3.27}.

Next, writing $[Q_j - (z - z_0) P_j] = I_{\cH} - [1 + (z - z_0)] P_j$, $1\leq j \leq n_0$, one can rewrite 
\eqref{3.26} in the form
\begin{equation}
A(z) = [I_{\cH} - F_0(z)] A_{n_0}(z), \quad z \in \Omega,       \lb{3.41} 
\end{equation} 
with $F_0(\cdot) \in \cF(\cH)$ analytic on $\Omega$. Similarly, writing 
$[Q_j - (z - z_0) P_j]^{-1} = Q_j - (z - z_0)^{-1} P_j = I_{\cH} - \big[1 + (z - z_0)^{-1}\big] P_j$, 
$1\leq j \leq n_0$, one obtains that 
\begin{equation}
[I_{\cH} - F_0(z)]^{-1} = I_{\cH} + F_1 (z), \quad z \in \cN(z_0), 
\end{equation}
for a sufficiently small neighborhood $\cN(z_0) \subset \Omega$ of $z_0$, and with 
$F_1(\cdot) \in \cF(\cH)$ meromorphic on $\cN(z_0)$ and analytic on $\cN(z_0) \backslash \{z_0\}$. 
Thus, one computes using \eqref{3.19} and \eqref{3.41} 
\begin{align}
\begin{split} 
& T(z) A(z) T(z)^{-1} = \begin{pmatrix} F(z) & 0 \\ 0 & G(z) \end{pmatrix}    \\
& \quad = \big[I_{\cH} - T(z) F_0(z) T(z)^{-1}\big] 
T(z) A_{n_0}(z) T(z)^{-1}, \quad z \in \cN(z_0), 
\end{split} 
\end{align}
and hence
\begin{align}
& \begin{pmatrix} F(z) & 0 \\ 0 & I_{\ran(Q(z_0))} \end{pmatrix} = T(z) A(z) T(z)^{-1} 
\begin{pmatrix} I_{\ran(P(z_0))} & 0 \\ 0 & G(z)^{-1} \end{pmatrix}    \no \\
& \quad = \big[I_{\cH} - T(z) F_0(z) T(z)^{-1}\big] T(z) A_{n_0}(z) T(z)^{-1} 
\begin{pmatrix} I_{\ran(P(z_0))} & 0 \\ 0 & G(z)^{-1} \end{pmatrix},    \lb{3.37} \\
& \hspace*{9.45cm} z \in \cN(z_0),    \no 
\end{align} 
implying 
\begin{align}
& T(z) A_{n_0}(z) T(z)^{-1} \begin{pmatrix} I_{\ran(P(z_0))} & 0 \\ 0 & G(z)^{-1} \end{pmatrix}  \no \\
& \quad = \big[I_{\cH} - T(z) F_0(z) T(z)^{-1}\big]^{-1}   
\begin{pmatrix} F(z) & 0 \\ 0 & I_{\ran(Q(z_0))} \end{pmatrix}    \no \\
& \quad = \big[I_{\cH} + T(z) F_1(z) T(z)^{-1}\big]   
\begin{pmatrix} I_{\ran(P(z_0))} + [F(z) - I_{\ran(P(z_0))}] & 0 \\ 0 & I_{\ran(Q(z_0))} \end{pmatrix}    \no \\
& \quad = [I_{\cH} + F_2(z)] [I_{\cH} - F_3(z)]    \no \\
& \quad = [I_{\cH} - F_4(z)], \quad z \in \cN(z_0).     \lb{3.38}
\end{align}
In \eqref{3.38}, we have set
\begin{align}
F_2(z) &= T(z) F_1(z) T(z)^{-1},\\
F_3(z) &= I_{\cH} - \begin{pmatrix} F(z) & 0 \\ 0 & I_{\ran(Q(z_0))} \end{pmatrix},\\[1mm] 
F_4(z) &= - F_2(z) + F_2(z)F_3(z) + F_3(z),\quad z\in \cN(z_0),
\end{align}
where $F_k(\cdot) \in \cF(\cH)$, $2 \leq k \leq 4$, $F_3(\cdot)$ is analytic on $\cN(z_0)$, 
$F_2(\cdot), F_4(\cdot)$ are meromorphic on $\cN(z_0)$ and analytic on $\cN(z_0) \backslash \{z_0\}$. 
In fact, since the left-hand side of \eqref{3.38} is analytic and boundedly invertible on $\cN(z_0)$, 
$F_4(\cdot)$ is analytic on $\cN(z_0)$ and 
\begin{equation}
[I_{\cH} - F_4(z)]^{-1} \in \cB(\cH), \quad {\det}_{\cH}(I_{\cH} - F_4(z)) \neq 0, \quad z \in \cN(z_0). 
\end{equation} 
Combining \eqref{3.37} and \eqref{3.38} 
then yields
\begin{align}
& \begin{pmatrix} F(z) & 0 \\ 0 & I_{\ran(Q(z_0))} \end{pmatrix} = 
\big[I_{\cH} - T(z) F_0(z) T(z)^{-1}\big] [I_{\cH} - F_4(z)]    \no \\
 & \quad = T(z) [Q_1 - (z - z_0) P_1] [Q_2 - (z - z_0) P_2] \cdots [Q_{n_0} - (z - z_0) P_{n_0}] 
T(z)^{-1}   \no \\
& \qquad \times [I_{\cH} - F_4(z)],  \quad z \in \cN(z_0). 
\end{align}
Thus, by \eqref{3.15}, 
\begin{equation}
{\det}_{\cH} \left(\begin{pmatrix} F(z) & 0 \\ 0 & I_{\ran(Q(z_0))} \end{pmatrix}\right) 
= \bigg(\prod_{j=1}^{n_0} (z_0 - z)^{p_j}\bigg) {\det}_{\cH}(I_{\cH} - F_4(z)), \quad z \in \cN(z_0), 
\end{equation}
and hence \eqref{3.32} holds. 

Relations \eqref{3.42} are clear from \eqref{3.14}. 
\end{proof} 

\begin{corollary} [\cite{Ho71}] \lb{c3.6}
Assume that $A:\Omega \to \cB(\cH)$ satisfies Hypothesis \ref{h3.2} and let $z_0 \in \Omega$ 
be a pole of $A(\cdot)^{-1}$. Then $z_0$ is a simple pole of $A(\cdot)^{-1}$ if and only if 
$\nu(z_0; A(\cdot)) = m_g(0; A(z_0))$.
\end{corollary}

In the remainder of this section we briefly derive the analogous factorizations in Theorems 
\ref{t3.4} and \ref{t3.5} but with the order of factors in \eqref{3.10a} and \eqref{3.26} interchanged. 
This appears to be a new result.

\begin{theorem} \lb{t3.7}
Assume that $A:\Omega \to \cB(\cH)$ satisfies Hypothesis \ref{h3.2} and let 
$z_0 \in \Omega$ be a pole of $A(\cdot)^{-1}$. Denote by $\wti P_1$ any projection 
onto $\ker(A(z_0))$ and let $\wti Q_1 = I_{\cH} - \wti P_1$. 
Then,
\begin{equation}
A(z) = \wti A_1(z) \big[\wti Q_1 - (z-z_0) \wti P_1\big], \quad z \in \Omega,     \lb{3.59A} 
\end{equation}
where 
\begin{align} 
& \text{$\wti A_1(\cdot)$ is analytic on $\Omega$,}      \lb{3.60a} \\
& \wti A_1(z) \in \Phi(\cH), \quad z\in\Omega,    \lb{3.61a} \\ 
& \de\big(\wti A_1(z_0)\big) \leq \de(A(z_0)),    \lb{3.62a} \\
& \ind\big(\wti A(z)\big) = \ind\big(\wti A_1(z)\big) = 0, 
\quad z \in \Omega, \; |z - z_0| \, \text{ sufficiently small.} \lb{3.62b}
\end{align}
If $z_0$ is a pole of $A(\cdot)^{-1}$ of order $n_0\in\bbN$, then $z_0$ is a pole 
of $\big(\wti A_1(\cdot)\big)^{-1}$ of order $n_0 - 1$. Finally, 
\begin{align}
& [I_{\cH} - A(\cdot)] \in \cF(\cH) \, \text{ $($resp., $\cB_p(\cH)$ for some $1 \leq p \leq \infty$$)$}  
\lb{3.63a} \\ 
& \quad \text{if and only if } \, 
\big[I_{\cH} - \wti A_1(\cdot)\big] \in \cF(\cH) \, \text{ $($resp., $\cB_p(\cH)$ for some $1 \leq p \leq \infty$$)$.}   \no 
\end{align}
\end{theorem}
\begin{proof} 
Set
\begin{equation}\lb{3.70B}
\ul{\Omega} = \{\ol{z} \in \bbC \, |\, z\in \Omega\},
\end{equation}
and define $B:\ul{\Omega}\rightarrow \cB(\cH)$ by
\begin{equation}\lb{3.71B}
B(\zeta) = A(\ol{\zeta})^*,\quad \zeta\in \ul{\Omega}.
\end{equation}
Evidently, $B(\zeta)\in \Phi(\cH)$, $\zeta\in \ul{\Omega}$, and $\ol{z_0}$ is a pole of $B(\,\cdot\,)^{-1}$.  If $\widetilde{P}_1$ is any projection onto $\ker(A(z_0))$ and $\widetilde{Q}_1 = I_{\cH} - \widetilde{P}_1$, then 
\begin{equation}\lb{3.72B}
\widetilde{Q}_1^* = I_{\cH} -\widetilde{P}_1^*
\end{equation}
projects onto (cf., e.g., \cite[p.~155--156, 267]{Ka80})
\begin{equation}\lb{3.73B}
[\ran(\widetilde{P}_1)]^{\perp} = [\ker(A(z_0))]^{\perp} = \ran(A(z_0)^*) = \ran(B(\ol{z_0})),
\end{equation} 
employing the closed range property of $A(z_0)$ and hence of $A(z_0)^*$ due to the Fredholm hypothesis on $A(\cdot)$ in \eqref{3.4}. 
Applying Theorem \ref{t3.4} to $B(\,\cdot\,)$ and $\widetilde{Q}_1^*$, one obtains an $A_1(\,\cdot\,):\ul{\Omega}\rightarrow \Phi(\cH)$ with the properties listed in Theorem \ref{t3.4}.  In particular,
\begin{equation}\lb{3.74B}
B(\zeta) = \big[\widetilde{Q}_1^* - (\zeta-\ol{z_0})\widetilde{P}_1^* \big]A_1(\zeta),\quad \zeta \in \ul{\Omega}.
\end{equation}
Taking adjoints in \eqref{3.74B} results in 
\begin{equation}\lb{3.75B}
A(\ol{\zeta}) = B(\zeta)^* 
= A_1(\zeta)^*\big[\widetilde{Q}_1 - (\ol{\zeta}-z_0)\widetilde{P}_1 \big],\quad \zeta \in \ul{\Omega},
\end{equation}
and since $\zeta \in \ul{\Omega}$ is equivalent to $\zeta=\ol{z}$ for some $z\in \Omega$, one obtains
\begin{equation}\lb{3.76B}
A(z) = A_1(\ol{z})^*\big[\widetilde{Q}_1 - (z - z_0)\widetilde{P_1} \big],\quad z\in \Omega.
\end{equation}
Therefore, \eqref{3.59A} holds with
\begin{equation}\lb{3.77B}
\widetilde{A}_1(z) := A_1(\ol{z})^*,\quad z\in \Omega,
\end{equation}
and $A_1(\,\cdot\,)$ inherits the properties \eqref{3.60a} and \eqref{3.62a} from the corresponding properties \eqref{3.12o} and \eqref{3.11}.  Employing $A(z_0)= \wti A_1(z_0) \wti Q_1$, one concludes that
\begin{equation}\lb{3.78B}
\ran(A(z_0)) \subseteq \ran\big(\wti A_1(z_0)\big), \quad \ran\big(\wti A_1(z_0)\big)^{\bot} 
\subseteq \ran (A(z_0))^{\bot},
\end{equation}
proving \eqref{3.62a}. Invariance of the Fredholm index as recorded in \eqref{3.62b} follows precisely as in the proof of Theorem \ref{t3.5}. The statement about orders of poles of $A(\,\cdot\,)^{-1}$ and $\widetilde{A}_1(\,\cdot\,)^{-1}$ is also clear.  If $z_0$ is a pole of $A(\,\cdot\,)^{-1}$ of order $n_0\in \bbN$, then $\ol{z_0}$ is a pole of $B(\,\cdot\,)^{-1}$ of order $n_0$.  Hence, $\ol{z_0}$ is a pole of $A_1(\,\cdot\,)^{-1}$ of order $n_0-1$ (by Theorem \ref{t3.4}), and $z_0$ is a pole of $\widetilde{A}_1(\,\cdot\,)^{-1}$ of order $n_0-1$.  Finally, \eqref{3.63a} follows from the corresponding statement \eqref{3.14} for $B(\,\cdot\,)$ and $A_1(\,\cdot\,)$ by taking adjoints. 
\end{proof}

Applying Theorem \ref{t3.5} to $B(\,\cdot\,)$ as defined in \eqref{3.71B}, one obtains the following analog of Theorem \ref{t3.5} with the order of factors reversed.

\begin{theorem} \lb{t3.8}
Assume that $A:\Omega \rightarrow \cB(\cH)$ satisfies Hypothesis \ref{h3.2}, and let $z_0\in \Omega$ be a pole of $A(\cdot)^{-1}$ of order $n_0\in \bbN$.  Then there exist projections $\wti P_j$ and $\wti Q_j = I_{\cH} - \wti P_j$ in $\cH$ such that with $\wti p_j = \dim(\ran(\wti P_j))$, $1\leq j \leq n_0$, one infers that
\begin{align}
A(z) &= \wti A_{n_0}(z)\big[\wti Q_{n_0} - (z-z_0)\wti P_{n_0}\big]\cdots \big[\wti Q_{2} - (z-z_0)\wti P_{2}\big]\big[\wti Q_{1} - (z-z_0)\wti P_{1}\big],\no\\
&\hspace*{9cm}z\in \Omega,\lb{3.74a}
\end{align}
and
\begin{equation}\lb{3.75a}
1\leq \wti p_{n_0} \leq \wti p_{n_0-1} \leq \cdots \leq \wti p_2 \leq \wti p_1 < \infty,
\end{equation}
where
\begin{align} 
& \text{$\wti A_{n_0}(\cdot)$ is analytic on $\Omega$,}      \lb{3.60aaa} \\
& \wti A_{n_0}(z) \in \Phi(\cH), \quad z\in\Omega,    \lb{3.61aaa} \\ 
& \ind\big(\wti A(z)\big) = \ind\big(\wti A_{n_0}(z)\big) = 0, 
\quad z \in \Omega, \; |z - z_0| \, \text{ sufficiently small,}     \lb{3.61b} \\
& \big[\wti A_{n_0}(z)\big]^{-1} \in \cB(\cH), \quad z \in \Omega, \; |z - z_0| \, 
\text{ sufficiently small.}    \lb{3.62aaa}
\end{align}
In addition,
\begin{equation}\lb{3.79a}
\wti p_1 = \dim(\ker (A(z_0))) = m_g(0;A(z_0)),
\end{equation}
and, hence, 
\begin{equation}\lb{3.80a}
\wti \nu(z_0;A(\cdot)) = \sum_{j=1}^{n_0} \wti p_j \geq m_g(0;A(z_0)),
\quad \wti \nu(z_0;A(\cdot))\geq n_0.
\end{equation}
Finally,
\begin{align}
& [I_{\cH} - A(\cdot)] \in \cF(\cH) \, \text{ $($resp., $\cB_p(\cH)$ for some $1 \leq p \leq \infty$$)$}  
\lb{3.63aaa} \\ 
& \quad \text{if and only if } \, 
\big[I_{\cH} - \wti A_{n_0}(\cdot)\big] \in \cF(\cH) \, \text{ $($resp., $\cB_p(\cH)$ for some $1 \leq p \leq \infty$$)$.}   \no 
\end{align}
\end{theorem}
\begin{proof}
Applying Theorem \ref{t3.4} to $B(\,\cdot\,)$ defined in \eqref{3.71B}, one obtains the existence of $A_{n_0}(\,\cdot\,)$ with the properties in Theorem \ref{t3.4} and projections $P_j$ and 
$Q_j = I_{\cH} - P_j$ in $\cH$ such that, with $\dim(\ran(P_j)) < \infty$, $1 \leq j \leq n_0$, the factorization
\begin{equation}
B(\zeta) = [Q_1 - (\zeta - \ol{z_0}) P_1] [Q_2 - (\zeta- \ol{z_0}) P_2] \cdots [Q_{n_0} - (\zeta - \ol{z_0}) P_{n_0}] A_{n_0}(\zeta), 
\quad \zeta \in \ul{\Omega},    \lb{3.26B}
\end{equation}
holds.
Taking adjoints in \eqref{3.26B} implies 
\begin{align}
A(\ol{\zeta}) &= B(\zeta)^*     \no \\
&= A_{n_0}(\zeta)^* [Q_{n_0}^* - (\ol{\zeta} - z_0) P_{n_0}^*]\cdots [Q_2^* - (\ol{\zeta}- z_0) P_2^*][Q_1^* - (\ol{\zeta} - z_0) P_1^*],      \no \\   
& \hspace*{9cm} \zeta\in \ul{\Omega}.\lb{3.88B}
\end{align}
Writing $\ol{\zeta}=z\in \Omega$, \eqref{3.88B} takes the form
\begin{equation}
A(z) = A_{n_0}(\ol{z})^* [Q_{n_0}^* - (z - z_0) P_{n_0}^*]\cdots [Q_2^* - (z- z_0) P_2^*][Q_1^* - (z - z_0) P_1^*],\quad z\in \Omega.\lb{3.89B}
\end{equation}
Defining $\widetilde{P_j}=P_j^*$, $1\leq j\leq n_0$, and 
\begin{equation}
\widetilde{A}_{n_0}(z) = A_{n_0}(\ol{z})^*,\quad z\in \Omega,
\end{equation}
yields the factorization \eqref{3.74a} with 
$\widetilde{p}_j = \dim(\ran(P_j)^*) = \dim(\ran(P_j))$, $1\leq j\leq n_0$.  The properties \eqref{3.60aaa}--\eqref{3.63aaa} follow immediately from the corresponding properties of 
$A_{n_0}(\,\cdot\,)$ and $B(\,\cdot\,)$.  We omit further details at this point, and only note that 
\begin{equation}
\widetilde{p}_1 = \dim(\ker(B(\ol{z_0}))) = \dim(\ker(A(z_0))).
\end{equation}
\end{proof}

\begin{remark} \lb{r3.9} 
Since $A(\cdot), A_1(\cdot), \dots , A_{n_0}(\cdot)$ in Theorems \ref{t3.4} and \ref{t3.5} 
all have index zero at $z_0 \in \Omega$, the dimensions $p_j$ in Theorem \ref{t3.5} and 
$\wti p_j$ in Theorem \ref{t3.8} satisfy $p_j = \wti p_j$, $1 \leq j \leq n_0$, repeatedly 
applying equalities of the type $p_1 = \dim(\ker(A(z_0)^*)) = \dim(\ker(A(z_0))) = \wti p_1$, etc. 
In particular, $\nu(z; A(\cdot)) = \wti \nu(z_0; A(\cdot))$. 
\end{remark}

\section{Algebraic Multiplicities of Zeros of Operator-Valued Functions: The Analytic Case}
\lb{s4}

In this section we consider algebraic multiplicities of zeros of analytic operator-valued functions 
and then study applications to the Birman--Schwinger operator $K(\cdot)$ in 
\eqref{2.4} in search of 
the analog of Theorem \ref{t2.7} for algebraic multiplicities.

The pertinent facts in this context can be found in \cite{GS71} (see 
also, \cite[Sects.\ XI.8, XI.9]{GGK90}, \cite[Ch.\ 4]{GL09}, and \cite[Sect.\ 11]{Ma88}). Let 
$\Omega \subseteq \bbC$ be open and connected, $z_0 \in \Omega$,  suppose that 
$A:\Omega \to \cB(\cH)$ is analytic on $\Omega$, and that 
$A(z_0)$ is a Fredholm operator (i.e., $\dim(\ker(A(z_0)))<\infty$, 
$\ran(A(z_0))$ is closed in $\cH$, and $\dim(\ker(A(z_0)^*))<\infty$) 
of index zero, that is, 
\begin{equation} 
\ind(A(z_0)) = \dim(\ker(A(z_0))) - \dim(\ker(A(z_0)^*)) = 0. 
\end{equation}
By \cite{GS71} (or by 
\cite[Theorem\ XI.8.1]{GGK90}) there exists a neighborhood 
$\cN(z_0) \subset \Omega$ and analytic and boundedly invertible 
operator-valued functions $E_j: \Omega \to \cB(\cH)$, $j=1,2$, such that 
\begin{equation}
A(z) = E_1(z) \wti A(z) E_2(z), \quad z \in \cN(z_0), 
\end{equation}
where $\wti A(\cdot)$ is of the particular form
\begin{equation}
\wti A(z) = \wti P_0 + \sum_{j=1}^r (z - z_0)^{n_j} \wti P_j, \quad z \in \cN(z_0),   \lb{3.51} 
\end{equation}
with 
\begin{align} 
& \wti P_k, \; 0 \leq k \leq r, \, \text{ mutually disjoint projections in $\cH$,}\no    \\ 
& \big[I_{\cH} - \wti P_0 \big] \in \cF(\cH), \quad 
\dim\big(\ran\big(\wti P_j\big)\big) = 1, \quad 1 \leq j \leq r,    \no \\ 
& n_1 \leq n_2 \leq \dots \leq n_r, \quad n_j \in \bbN, \; 1 \leq j \leq r.     \lb{3.54}
\end{align} 
Moreover (cf.\ \cite[Sect.\ XI.8]{GGK90}, \cite{GS71}), the integers $n_j$, $1\leq j \leq r$, are 
uniquely determined by $A(\cdot)$, and the geometric multiplicity 
$m_g(0; A(z_0))$ of $A(z_0)$ is given by
\begin{equation}
m_g(0; A(z_0)) = \dim(\ker(A(z_0))) 
= \dim\big(\ran\big(I_{\cH} - \wti P_0\big)\big). 
\end{equation}

\begin{definition} \lb{d3.7}
Let $\Omega \subseteq \bbC$ be open and connected, $z_0 \in \Omega$,  suppose that 
$A:\Omega \to \cB(\cH)$ is analytic on $\Omega$. Then $z_0$ is called a {\it zero of finite-type 
of $A(\cdot)$} if $A(z_0)$ is a Fredholm operator, $\ker(A(z_0)) \neq \{0\}$, and $A(\cdot)$ is 
boundedly invertible on $D(z_0; \varepsilon_0) \backslash \{z_0\}$, for sufficiently small 
$\varepsilon_0 > 0$. 
\end{definition}

Our choice of notation calling $z_0$ a {\it zero} of $A(\cdot)$ is close to Howland's notation 
of a {\it weak zero} of an operator-valued function in \cite{Ho71}. In Gohberg and Sigal \cite{GS71} 
(and in related literature in the former Soviet Union) the notion of a {\it characteristic value} 
(or {\it eigenvalue}) is used instead in this connection. 

If $z_0$ is a zero of finite-type of $A(\cdot)$, then (since $A(\cdot)$ is 
boundedly invertible on $D(z_0; \varepsilon_0) \backslash \{z_0\}$, for sufficiently small 
$\varepsilon_0 > 0$), 
\begin{equation}
\ind(A(z_0)) = 0,  
\end{equation}
and one has 
\begin{equation}
\sum_{k=0}^r \wti P_k = I_{\cH}, 
\end{equation} 
and 
\begin{equation} 
\text{$A(\cdot)^{-1}$ is finitely meromorphic at $z_0$}  
\end{equation} 
(cf.\ \cite[Sect.\ XI.9]{GGK90}, \cite{GS71} for these facts).

\begin{definition} \lb{d3.8}
Let $\Omega \subseteq \bbC$ be open and connected, $z_0 \in \Omega$,  suppose that 
$A:\Omega \to \cB(\cH)$ is analytic on $\Omega$, and assume that $z_0$ is a zero of finite-type 
of $A(\cdot)$. Then $m_a(z_0; A(\cdot))$, the {\it algebraic multiplicity of the zero of $A(\cdot)$ 
at $z_0$}, is defined to be (cf.\ \cite[Sect.\ XI.9]{GGK90})
\begin{equation}
m_a(z_0; A(\cdot)) = \sum_{j=1}^r n_j,  
\end{equation} 
with $n_j$, $1 \leq j \leq r$, introduced in \eqref{3.51}--\eqref{3.54}. 
\end{definition}

Under the assumptions in Definition \ref{d3.8}, one also has an extension of the argument principle for scalar analytic functions to the operator-valued case (cf.\ 
\cite[Theorem\ XI.9.1]{GGK90}, \cite{GS71}) in the form
\begin{align}
\begin{split}
m_a(z_0; A(\cdot)) &= {\tr}_{\cH}\bigg(\f{1}{2\pi i} 
\ointctrclockwise_{C(z_0; \varepsilon)} d\zeta \, 
A'(\zeta) A(\zeta)^{-1}\bigg)     \lb{3.53} \\
& = {\tr}_{\cH}\bigg(\f{1}{2\pi i} 
\ointctrclockwise_{C(z_0; \varepsilon)} d\zeta \, 
A(\zeta)^{-1} A'(\zeta)\bigg), \quad 0 < \varepsilon < \varepsilon_0.
\end{split}
\end{align}  
Since $A(\cdot)^{-1}$ is finitely meromorphic, the integral in \eqref{3.53} 
is a finite rank operator (the analytic and non-finite-rank part under the integral in \eqref{3.53} yielding a zero contribution when integrated over 
$C(z_0; \varepsilon)$) and hence the trace in \eqref{3.53} is 
well-defined. 

Next, recalling our notation of the principal part of an operator-valued meromorphic function in 
\eqref{3.2}, one also obtains 
\begin{align}
\begin{split}
m_a(z_0; A(\cdot)) &= {\tr}_{\cH}\bigg(\f{1}{2\pi i} 
\ointctrclockwise_{C(z_0; \varepsilon)} d\zeta \, 
{\rm pp}_{z_0} \, \big\{A'(\zeta) A(\zeta)^{-1}\big\}\bigg)    \\ 
&= {\tr}_{\cH}\bigg(\f{1}{2\pi i} 
\ointctrclockwise_{C(z_0; \varepsilon)} d\zeta \, 
{\rm pp}_{z_0} \, \big\{A(\zeta)^{-1}A'(\zeta) \big\}\bigg), 
\quad 0 < \varepsilon < \varepsilon_0.       \lb{3.56}
\end{split}
\end{align}  

Moreover, we mention the following useful result.

\begin{lemma} [{\cite[Lemma\ 9.3]{GGK90}}, {\cite[Proposition\ 4.2.2]{GL09}}] 
\lb{l3.9}
Let $\Omega\subseteq\bbC$ be open and connected and $M_j(\cdot)$, $j=1,2$, be 
finitely meromorphic at $z_0 \in \Omega$. Then 
$M_1(\cdot) M_2(\cdot)$ and $M_2(\cdot) M_1(\cdot)$ are finitely meromorphic at 
$z_0 \in \Omega$, and for $0 < \varepsilon_0$ sufficiently small,
\begin{align}
& \ointctrclockwise_{C(z_0;\varepsilon)} d \zeta \, M_1(\zeta) M_2(\zeta), 
\; \ointctrclockwise_{C(z_0;\varepsilon)} d \zeta \, M_2(\zeta) M_1(\zeta) 
\in \cF(\cH),    \lb{3.57} \\
& {\tr}_{\cH} \bigg(\ointctrclockwise_{C(z_0;\varepsilon)} d \zeta \, 
M_1(\zeta) M_2(\zeta)\bigg) 
= {\tr}_{\cH} \bigg(\ointctrclockwise_{C(z_0;\varepsilon)} d \zeta \, 
M_2(\zeta) M_1(\zeta)\bigg),    \lb{3.58} \\
& {\tr}_{\cH} \big({\rm pp}_{z_0} \, \{M_1(z) M_2(z)\}\big) = 
{\tr}_{\cH} \big({\rm pp}_{z_0} \, \{M_2(z) M_1(z)\}\big), \quad 
0 < |z-z_0| < \varepsilon_0.     \lb{3.59}
\end{align}
\end{lemma}

\medskip

One notes that $m_a(z_0; A(\cdot))$ must be distinguished from $m_a(0; A(z_0))$. However, in 
the special case where $A(z) = A - z I_{\cH}$, $z \in \Omega$, one has the following fact:

\begin{remark} \lb{r3.10}
Let $\Omega \subseteq \bbC$ be open and connected, $z_0 \in \Omega$, and suppose that the particular 
function $A(z) = A - z I_{\cH}$, $z \in \Omega$, with $A \in \cB(\cH)$, satisfies the conditions of 
Definition \ref{d3.8}. Then $m_a(z_0; A(\cdot))$, the {\it algebraic multiplicity of the 
zero of $A(\cdot)$ at $z_0$} equals the algebraic multiplicity $m_a(z_0;A)$ of the eigenvalue $z_0$ 
of $A$, 
\begin{equation}
m_a(z_0; A(\cdot)) = m_a(z_0; A),   
\end{equation} 
since the right-hand side of \eqref{3.53} equals
\begin{equation}
{\tr}_{\cH}\bigg(\f{-1}{2\pi i} 
\ointctrclockwise_{C(z_0; \varepsilon)} d\zeta \, 
(A - \zeta I_{\cH})^{-1}\bigg) = {\tr}_{\cH}(P(z_0;A)) = m_a(z_0; A), \quad 0 < \varepsilon < \varepsilon_0, 
\end{equation}
with $P(z_0;A)$ the Riesz projection associated with $A$ and its isolated eigenvalue $z_0$. 
\end{remark}

Following standard practice, we now introduce the {\it discrete spectrum} of a densely defined, closed, linear 
operator $T$ in $\cH$ by
\begin{equation}
\sigma_d(T) = \{z \in \sigma_p(T) \,|\, \text{$z$ is an isolated point of $\sigma(T)$, with $m_a(z;T)<\infty$}\}, 
\lb{3.59a}
\end{equation}
and denote its {\it essential spectrum} by  
\begin{equation}
\sigma_{ess}(T) = \bbC \backslash \sigma_d(T).    \lb{3.59b} 
\end{equation}
Since $T$ is not assumed to be self-adjoint, one notes that several inequivalent definitions of 
$\sigma_{ess}(T)$ are in use in the literature (cf., e.g., \cite[Ch.\ IX]{EE89} for a detailed discussion), but 
in this paper we will only use the one in \eqref{3.59b}. 

Given this background material, we now apply it to reprove and slightly extend a multiplicity 
result due to Latushkin and Sukhtyaev \cite{LS10}.  

\begin{theorem} \lb{t3.11}
Assume Hypothesis \ref{h2.5} and suppose that 
$z_0 \in \rho(H_0) \cap \sigma(H)$ with $D(z_0; \varepsilon_0) \cap 
\sigma(H) = \{z_0\}$ for some $\varepsilon_0 >0$. Then $z_0$ is a discrete eigenvalue of $H$, 
\begin{equation}
z_0 \in \sigma_d(H).     \lb{3.59c}
\end{equation}
In addition, $z_0$ is a zero of finite-type of $I_{\cK} - K(\cdot)$ and 
\begin{equation}
m_a(z_0; H) = m_a(z_0; I_{\cK} - K(\cdot))
= \nu(z_0; I_{\cK} - K(\cdot)).      \lb{3.60}
\end{equation}
\end{theorem} 
\begin{proof}\footnote{We slightly corrected and extended the first paragraph of this proof.}
By \eqref{2.13}, any singularity $z_0 \in \rho(H_0)$ of 
$R(z) = (H-zI_{\cH})^{-1}$ must be a singularity of 
$[I_{\cK} - K(\cdot)]^{-1}$. Since by Hypothesis \ref{h2.5}, $[I_{\cK} - K(\cdot)] \in \Phi(\cK)$ 
on $\rho(H_0)$ (and of course $[I_{\cK} - K(\cdot)]$ is analytic on $\rho(H_0)$), and by \eqref{2.25}, $[I_{\cK} - K(z_1)]$ is boundedly invertible for some 
$z_1 \in D(z_0; \varepsilon_1) \backslash \{z_0\} \subset \rho(H_0)$ for $0 < \varepsilon_1$ sufficiently small, 
$\varepsilon_1 < \varepsilon_0$, alternative $(ii)$ in Theorem \ref{t3.3}, with $\Omega = D(z_0;\varepsilon_1)$, applies to $[I_{\cK} - K(\cdot)]$. 
In particular, $[I_{\cK} - K(\cdot)]^{-1}$ is finitely meromorphic on $D(z_0;\varepsilon_1)$ and hence 
so is $R(\cdot)$. By \cite[Sect.\ III.6.5]{Ka80}, this implies that $z_0 \in \sigma_p(H)$ and then again by the finitely meromorphic property of $R(\cdot)$ on $D(z_0;\varepsilon_1)$, the Riesz projection associated with $z_0$, 
\begin{equation}
P(z_0;H)=\f{-1}{2\pi i} \ointctrclockwise_{\cC(z_0;\varepsilon)} 
dz \, (H - z I_{\cH})^{-1},   \quad 0 < \varepsilon < \varepsilon_1,   \lb{3.61}
\end{equation}
is finite-dimensional, which in turn is equivalent to the eigenvalue $z_0$ having finite algebraic multiplicity 
and hence to \eqref{3.59c}.

Without loss of generality we may assume that $z_0=0$ for the remainder of the proof of 
Theorem \ref{t3.11}. Identifying $A(\cdot) = I_{\cK} - K(\cdot)$, an application of 
Theorem~\ref{t3.5} (using the notation employed in the latter) yields for $0 < |z| < \varepsilon_0$, 
\begin{equation}
A(z) = [Q_1 - z P_1][Q_2 - z P_2] \cdots 
[Q_{n_0} - z P_{n_0}] A_{n_0} (z), 
\end{equation}
and 
\begin{equation}
\nu(0; A(\cdot)) = \sum_{j=1}^{n_0} p_j, \quad p_j = \dim(\ran(P_j)), \quad   
Q_j = I_{\cH} - P_j, \; 1 \leq j \leq n_0.   
\end{equation}
Thus, one computes
\begin{align}
 A(z)^{-1} A'(z) & = [A_{n_0} (z)]^{-1} [Q_{n_0} - z P_{n_0}]^{-1} 
\cdots [Q_1 - z P_1]^{-1}    \no \\
& \quad \times 
\big\{[-P_1][Q_2 - z P_2] \cdots [Q_{n_0} - z P_{n_0}] 
A_{n_0} (z)    \no \\
& \qquad + [Q_1 - z P_1] [-P_2] \cdots [Q_{n_0} - z P_{n_0}] A_{n_0} (z) 
\no \\
& \qquad + \cdots + 
[Q_1 - z P_1] [Q_2 - z P_2] \cdots [Q_{n_0 - 1} - z P_{n_0 - 1}] 
[- P_{n_0}] A_{n_0} (z)    \no \\
& \qquad + [Q_1 - z P_1] [Q_2 - zP_2] \cdots [Q_{n_0} - z P_{n_0}] 
A_{n_0}' (z) \big\}.     \lb{4.24aaa} 
\end{align}
Next, continuing the computation in \eqref{4.24aaa}, one infers
\begin{align}
&A(z)^{-1}A'(z)\no\\
&\quad = [A_{n_0} (z)]^{-1} A_{n_0}' (z) \lb{3.64}   \\
& \qquad + [A_{n_0} (z)]^{-1} [Q_{n_0} - z P_{n_0}]^{-1} 
\cdots [Q_1 - z P_1]^{-1}    \no \\
& \qquad \times 
\big\{[-P_1][Q_2 - z P_2] \cdots [Q_{n_0} - z P_{n_0}] 
A_{n_0} (z)    \no \\
& \qquad + [Q_1 - z P_1] [-P_2] \cdots [Q_{n_0} - z P_{n_0}] A_{n_0} (z) 
\no \\
& \qquad + \cdots + 
[Q_1 - z P_1] [Q_2 - z P_2] \cdots [Q_{n_0 - 1} - z P_{n_0 - 1}] 
[- P_{n_0}] A_{n_0} (z)\big\}.    \no 
\end{align}
Since the first term on the right-hand side of \eqref{3.64} is analytic at $z_0=0$, its contour 
integral over $C(0;\varepsilon)$, 
$0 < \varepsilon < \varepsilon_1$, vanishes and one obtains upon 
repeatedly applying cyclicity of the trace (i.e., 
${\tr}_{\cH}(CD) = {\tr}_{\cH}(DC)$ for $C, D \in \cB(\cH)$, with 
$CD, DC \in \cB_1(\cH)$),   
\begin{align}
& m_a(0; A(\cdot)) = {\tr}_{\cK} \bigg(\f{1}{2 \pi i} 
\ointctrclockwise_{C(0; \varepsilon)} d\zeta \, 
A(\zeta)^{-1} A'(\zeta) \bigg)  \no \\
& \quad = {\tr}_{\cK} \bigg(\f{1}{2 \pi i} 
\ointctrclockwise_{C(0; \varepsilon)} d\zeta \, 
[A_{n_0} (\zeta)]^{-1} [Q_{n_0} - \zeta P_{n_0}]^{-1}
\cdots [Q_1 - \zeta P_1]^{-1}    \no \\
& \qquad \qquad \times 
\big\{[-P_1][Q_2 - \zeta P_2] \cdots [Q_{n_0} - \zeta P_{n_0}] 
A_{n_0} (\zeta)    \no \\
& \qquad \qquad \quad \;\; + [Q_1 - \zeta P_1] [-P_2] \cdots 
[Q_{n_0} - \zeta P_{n_0}] A_{n_0} (\zeta)     \no \\
& \qquad \qquad \quad \;\; + \cdots + 
[Q_1 - \zeta P_1] [Q_2 - \zeta P_2] \cdots [Q_{n_0 - 1} - \zeta P_{n_0 - 1}] 
[- P_{n_0}] A_{n_0} (\zeta)\big\}\bigg)     \no \\
& \quad = \f{1}{2 \pi i} 
\ointctrclockwise_{C(0; \varepsilon)} d\zeta \, 
{\tr}_{\cK} \Big([A_{n_0} (\zeta)]^{-1} [Q_{n_0} - \zeta P_{n_0}]^{-1}
\cdots [Q_1 - \zeta P_1]^{-1}    \no \\
& \qquad \qquad \times 
\big\{[-P_1][Q_2 - \zeta P_2] \cdots [Q_{n_0} - \zeta P_{n_0}] 
A_{n_0} (\zeta)    \no \\
& \qquad \qquad \quad \;\; + [Q_1 - \zeta P_1] [-P_2] \cdots 
[Q_{n_0} - \zeta P_{n_0}] A_{n_0} (\zeta)     \no \\
& \qquad \qquad \quad \;\; + \cdots + 
[Q_1 - \zeta P_1] [Q_2 - \zeta P_2] \cdots [Q_{n_0 - 1} - \zeta P_{n_0 - 1}] 
[- P_{n_0}] A_{n_0} (\zeta)\big\}\Big)     \no \\
& \quad = \f{1}{2 \pi i} 
\ointctrclockwise_{C(0; \varepsilon)} d\zeta \, 
{\tr}_{\cK} \Big([Q_{n_0} - \zeta P_{n_0}]^{-1}
\cdots [Q_1 - \zeta P_1]^{-1}    \no \\
& \qquad \qquad \times 
\big\{[-P_1][Q_2 - \zeta P_2] \cdots [Q_{n_0} - \zeta P_{n_0}]   
+ [Q_1 - \zeta P_1] [-P_2] \cdots 
[Q_{n_0} - \zeta P_{n_0}]     \no \\
& \qquad \qquad \quad \;\; + \cdots + 
[Q_1 - \zeta P_1] [Q_2 - \zeta P_2] \cdots [Q_{n_0 - 1} - \zeta P_{n_0 - 1}] 
[- P_{n_0}]\big\}\Big)     \no \\ 
& \quad = \f{1}{2 \pi i} 
\ointctrclockwise_{C(0; \varepsilon)} d\zeta \, 
\sum_{j=1}^{n_0} \, {\tr}_{\cK}\big([Q_j - \zeta P_j]^{-1} [- P_j]\big)     \no \\
& \quad = \f{1}{2 \pi i} 
\ointctrclockwise_{C(0; \varepsilon)} d\zeta \, 
\sum_{j=1}^{n_0} \, {\tr}_{\cK}\big([Q_j - \zeta^{-1} P_j] [- P_j]\big)     \no \\
& \quad = \f{1}{2 \pi i} 
\ointctrclockwise_{C(0; \varepsilon)} d\zeta \, 
\bigg(\sum_{j=1}^{n_0} \, {\tr}_{\cK} (P_j) \bigg) \zeta^{-1}    \no \\ 
& \quad = \sum_{j=1}^{n_0} p_j = \nu(0;A(\cdot)).     \lb{3.65} 
\end{align}
Next, one computes 
\begin{align}
& m_a(0; I_{\cK} - K(\cdot)) = \f{1}{2 \pi i} {\tr}_{\cK}\bigg( 
\ointctrclockwise_{C(0; \varepsilon)} d\zeta \, 
[I_{\cK} - K(\zeta)]^{-1}[- K'(\zeta)]\bigg)     \no \\
& \quad = \f{1}{2 \pi i} {\tr}_{\cK}\bigg( 
\ointctrclockwise_{C(0; \varepsilon)} d\zeta \, 
[I_{\cK} - K(\zeta)]^{-1} V_1 R_0(\zeta) \ol{R_0(\zeta) V_2^*}\bigg)  \no \\ 
& \quad = \f{1}{2 \pi i} {\tr}_{\cH}\bigg( 
\ointctrclockwise_{C(0; \varepsilon)} d\zeta \, 
\ol{R_0(\zeta) V_2^*} [I_{\cK} - K(\zeta)]^{-1} V_1 R_0(\zeta)\bigg),  \lb{3.66} 
\end{align}
where we used \eqref{3.58} in the last step in \eqref{3.66}. Similarly, 
using \eqref{2.13}, 
\begin{align}
m_a(0;H) &= \f{-1}{2 \pi i} {\tr}_{\cH} \bigg(
\ointctrclockwise_{C(0; \varepsilon)} d\zeta \, 
(H - \zeta I_{\cH})^{-1}\bigg)    \no \\
&= \f{-1}{2 \pi i} {\tr}_{\cH} \bigg(
\ointctrclockwise_{C(0; \varepsilon)} d\zeta \, 
\big[(H - \zeta I_{\cH})^{-1} - (H_0 - \zeta I_{\cH})^{-1}\big] \bigg)    \no \\
&= \f{1}{2 \pi i} {\tr}_{\cH} \bigg(
\ointctrclockwise_{C(0; \varepsilon)} d\zeta \, 
\ol{R_0(\zeta) V_2^*} [I_{\cK} - K(\zeta)]^{-1} V_1 R_0(\zeta)\bigg).  \lb{3.67} 
\end{align}
Combining \eqref{3.65}--\eqref{3.67} proves \eqref{3.60}. 
\end{proof}

We note that the first equality in \eqref{3.60} is due to Latushkin and Sukhtyaev \cite{LS10} (our 
proof seems slightly shorter); the second equality in \eqref{3.60} is new. 

\begin{remark} \lb{r3.12}
It is amusing to note that in the special finite-dimensional case, $\cH = \cK$, 
$\dim(\cH) < \infty$, the following special case of \eqref{3.60}, namely,  
\begin{equation}
m_a(z_0; H) = \nu(z_0;  I_{\cH} - K(\cdot)), 
\quad z_0 \in \rho(H_0) \cap \sigma(H),  
\end{equation}
has basically a one-line proof! Indeed, in the matrix-valued context, the careful symmetrization in \eqref{2.13} becomes unnecessary and so abbreviating 
\begin{equation} 
V = V_2^* V_1,  \quad  K(z) = - V(H_0 - z I_{\cH})^{-1}, \quad z \in \rho(H_0), 
\end{equation} 
one obtains
\begin{align}
\begin{split} 
(H - z I_{\cH}) (H_0 - z I_{\cH})^{-1} 
&= (H_0 + V - z I_{\cH}) (H_0 - z I_{\cH})^{-1}     \\ 
&= I_{\cH} + V (H_0 - z I_{\cH})^{-1},   \quad z \in \rho(H_0).
\end{split}   
\end{align}
Thus, the underlying perturbation determinant becomes a ratio of determinants,  
\begin{align}
{\det}_{\cH} \big(I_{\cH} + V (H_0 - z I_{\cH})^{-1}\big) 
= \f{{\det}_{\cH} (H-z I_{\cH})}{{\det}_{\cH} (H_0-z I_{\cH})}, 
\quad z \in \rho(H_0), 
\end{align}
implying 
\begin{align}
\nu(z_0; I_{\cH} - K(\cdot)) &= m(z_0; {\det}_{\cH}(I_{\cH} - K(\cdot)))  \no \\
&= m_a(z_0;H) - m_a(z_0;H_0)   \no \\
&= m_a(z_0;H),   \quad z_0 \in \rho(H_0) \cap \sigma(H). 
\end{align}
\end{remark}

At first sight, \eqref{2.33} together with \eqref{3.60} appear to describe an incomplete picture since we did not yet invoke a discussion of the quantity 
$m_a(1; K(z_0))$. However, the following elementary two-dimensional examples show that $m_a(1; K(z_0))$ in general differs from the value determined in \eqref{3.60}: 

\begin{example} \lb{e3.13} Let $\cH = \cK = \bbC^2$.  \\
$(i)$ Introduce
\begin{equation}
H_0 = \begin{pmatrix} 1 & 1 \\ 0 & 1 \end{pmatrix}, \quad 
V = \begin{pmatrix} 0 & 0 \\ -1 & -2 \end{pmatrix}, \quad 
H= H_0 + V = \begin{pmatrix} 1 & 1 \\ -1 & -1 \end{pmatrix}.
\end{equation}
Then, $\sigma(H_0) = \{1\}$, $0 \in \sigma(H)$, and 
\begin{align}
& K(z) = - V (H_0 - z I_2)^{-1} 
= \f{-1}{(1-z)^2} \begin{pmatrix} 0 & 0 \\ - (1-z) & 1 - 2 (1-z) \end{pmatrix}, 
\quad z \in \bbC\backslash\{1\},   \no \\
& K(0) =  \begin{pmatrix} 0 & 0 \\ 1 & 1 \end{pmatrix},    \no \\
& {\det}_{\bbC^2}(H - z I_2) = z^2, \quad 
{\det}_{\bbC^2}(I_2 - K(z)) = z^2 (1 - z)^{-2},     \\
& {\det}_{\bbC^2}(z I_2 - K(0)) = - (1-z)\big[1 - (1-z)\big],    \no \\
& m_a(0; H) = \nu(0;I_2 - K(\cdot)) = 2, \quad 
m_g(0; H) = m_g(1; K(0)) = m_a(1; K(0)) = 1.    \no
\end{align}
$(ii)$ Introduce
\begin{equation}
H_0 = \begin{pmatrix} 1 & 1 \\ 0 & 1 \end{pmatrix}, \quad 
V = \begin{pmatrix} 0 & -1 \\ 1 & -1 \end{pmatrix}, \quad 
H= H_0 + V = \begin{pmatrix} 1 & 0 \\ 1 & 0 \end{pmatrix}.
\end{equation}
Then, $\sigma(H_0) = \{1\}$, $0 \in \sigma(H)$, and 
\begin{align}
& K(z) = - V (H_0 - z I_2)^{-1} 
= \f{-1}{(1-z)^2} \begin{pmatrix} 0 & - (1-z) \\ (1-z) & - (2-z) 
\end{pmatrix}, 
\quad z \in \bbC\backslash\{1\},    \no \\
& K(0) =  \begin{pmatrix} 0 & 1 \\ -1 & 2 \end{pmatrix},     \no \\
& {\det}_{\bbC^2}(H - z I_2) = -z(1-z), \quad 
{\det}_{\bbC^2}(I_2 - K(z)) = -z (1 - z)^{-1},     \\
& {\det}_{\bbC^2}(z I_2 - K(0)) = (1-z)^2,    \no \\
& m_g(0; H) = m_a(0; H) = \nu(0;I_2 - K(\cdot)) = m_g(1; K(0)) = 1, 
\quad m_a(1; K(0)) = 2.    \no
\end{align}
\end{example} 

We note that everything in this section applies to the Banach space setting, as is clear 
from the treatments in \cite[Ch.\ 4]{GL09}, \cite{GS71}, \cite{Ho71}. We should also note that 
notions of algebraic multiplicities of parts of the spectrum of operator-valued functions, differing 
from the one employed in the present paper, were used in \cite[Part~II]{LM07}, \cite{Ma95}, \cite{Ma97}.

\section{Algebraic Multiplicities of Operator-Valued Functions: \\ The Meromorphic Case}
\lb{s5}

The principal purpose of this section is to properly extend Theorem \ref{t3.11} to the case 
where $K(\cdot)$ is finitely meromorphic and \eqref{3.60} represents the analog of the 
Weinstein--Aronszajn formula in the sense that $H$ and $H_0$ have common discrete eigenvalues. 

\begin{hypothesis} \lb{h4.1} 
Let $\Omega \subseteq \bbC$ be open and connected, and $\cD_0 \subset \Omega$ a discrete set 
(i.e., a set without limit points in $\Omega$). Suppose that $M:\Omega \backslash \cD_0 \to \cB(\cH)$ is analytic and that 
$M(\cdot)$ is finitely meromorphic on $\Omega$. In addition, suppose that  
\begin{equation}
M(z) \in \Phi(\cH) \, \text{ for all } \, z\in\Omega \backslash \cD_0,    \lb{4.1} 
\end{equation}  
and for all $z_0 \in \cD_0$, such that 
\begin{equation}
M(z) = \sum_{k=-N_0}^{\infty} (z-z_0)^k M_k(z_0), \quad 
0 < |z - z_0| < \varepsilon_0,    \lb{4.2}
\end{equation}
for some $N_0 = N_0(z_0) \in \bbN$ and some $0 < \varepsilon_0 = \varepsilon_0(z_0)$ 
sufficiently small, with  
\begin{align}\lb{4.2a} 
\begin{split} 
& M_{-k}(z_0) \in \cF(\cH), \; 1 \leq k \leq N_0(z_0),  \quad  M_0(z_0) \in \Phi(\cH), \\
& M_k(z_0) \in \cB(\cH), \; k \in \bbN. 
\end{split} 
\end{align} 
\end{hypothesis}

One then recalls the meromorphic Fredholm theorem in the following form:

\begin{theorem} [{\cite[Proposition~4.1.4]{GL09}}, \cite{GS71}, \cite{Ho70}, {\cite[Theorem\ XIII.13]{RS78}}, \cite{RV69}, \cite{St68}] \lb{t4.2} ${}$ \\
Assume that $M:\Omega\backslash\cD_0 \to \cB(\cH)$ satisfies Hypothesis \ref{h4.1}. Then either \\
$(i)$ $M(z)$ is not boundedly invertible for any $z \in \Omega\backslash\cD_0$, \\[1mm]
or else, \\[1mm]
$(ii)$ $M(\cdot)^{-1}$ is finitely meromorphic on $\Omega$. More precisely, there exists a discrete 
subset $\cD_1 \subset \Omega$ $($possibly, $\cD_1 = \emptyset$$)$ such that 
$M(z)^{-1} \in \cB(\cH)$ for all $z \in \Omega\backslash\{\cD_0 \cup \cD_1\}$, $M(\cdot)^{-1}$ extends to 
an analytic function on $\Omega\backslash\cD_1$, meromorphic on $\Omega$. In addition, 
\begin{equation}\lb{4.5a}
M(z)^{-1} \in \Phi(\cH) \, \text{ for all } \, z \in \Omega\backslash\cD_1, 
\end{equation}
and if $z_1 \in \cD_1$ then for some $N_0(z_1) \in \bbN$, and for some $0 < \varepsilon_0(z_1)$ 
sufficiently small, 
\begin{equation}
M(z)^{-1} = \sum_{k= - N_0(z_1)}^{\infty} (z - z_1)^{k} D_k(z_1), \quad 
0 < |z - z_1| < \varepsilon_0(z_1),    \lb{4.5}
\end{equation}
with  
\begin{align} \lb{4.5b} 
\begin{split} 
& D_{-k}(z_1) \in \cF(\cH), \;1 \leq k \leq N_0(z_1),  \quad  D_0(z_1) \in \Phi(\cH), \\
& D_k(z_1) \in \cB(\cH), \; k \in \bbN. 
\end{split} 
\end{align}
In addition, if $[I_{\cH} - M(z)] \in \cB_{\infty}(\cH)$ for all $z \in \Omega\backslash\cD_0$, then 
\begin{equation}
\big[I_{\cH} - M(z)^{-1}\big] \in \cB_{\infty}(\cH), \quad z \in \Omega\backslash\cD_1, 
\quad [I_{\cH} - D_0(z_1)] \in \cB_{\infty}(\cH), \quad z_1 \in \cD_1.  
\end{equation} 
\end{theorem}

Next we strengthen Hypothesis \ref{h4.1} as follows: 

\begin{hypothesis} \lb{h4.3} 
Suppose $M(\cdot):\Omega\backslash\cD_0 \to \cB(\cH)$ satisfies Hypothesis \ref{h4.1}, let 
$z_0\in \cD_0$, and assume that $M(\cdot)$ is boundedly invertible on 
$D(z_0; \varepsilon_0) \backslash \{z_0\}$ for some $0 < \varepsilon_0$ sufficiently small. 
\end{hypothesis}

\begin{definition} \lb{d4.4} 
Assume Hypothesis \ref{h4.3}. Then the {\it index of $M(\cdot)$ with respect to the counterclockwise 
oriented circle $C(z_0; \varepsilon)$}, $\ind_{C(z_0; \varepsilon)}(M(\cdot))$, is defined by 
 \begin{align}
\begin{split}
\ind_{C(z_0; \varepsilon)}(M(\cdot)) &= {\tr}_{\cH}\bigg(\f{1}{2\pi i} 
\ointctrclockwise_{C(z_0; \varepsilon)} d\zeta \, 
M'(\zeta) M(\zeta)^{-1}\bigg)     \lb{4.8} \\
& = {\tr}_{\cH}\bigg(\f{1}{2\pi i} 
\ointctrclockwise_{C(z_0; \varepsilon)} d\zeta \, 
M(\zeta)^{-1} M'(\zeta)\bigg), \quad 0 < \varepsilon < \varepsilon_0. 
\end{split}
\end{align}  
\end{definition}

By the operator-valued version of the argument principle proved in \cite{GS71} (see also 
\cite[Theorem\ 4.4.1]{GL09}), one has 
 \begin{equation}
\ind_{C(z_0; \varepsilon)}(M(\cdot)) \in \bbZ     \lb{4.9}
\end{equation}  
under the conditions of Definition \ref{d4.4}. 
 
The next result presents our generalization of Theorem \ref{t3.11} to the case 
where $K(\cdot)$ is finitely meromorphic. In particular, we will now prove the analog of the 
Weinstein--Aronszajn formula (cf., e.g., \cite{AB70}, \cite{Ho70}, \cite[Sect.\ IV.6]{Ka80}, \cite{Ku61}, 
\cite[Sect.\ 9.3]{WS72}) in the case where $H$ and $H_0$ have common discrete eigenvalues. 

\begin{theorem} \lb{t4.5}
Assume Hypothesis \ref{h2.5} and suppose that $z_0 \in (\sigma_d(H_0) \cap \sigma(H))$ with 
$D(z_0; \varepsilon_0) \cap \sigma(H) = \{z_0\}$ for some $0 < \varepsilon_0 = \varepsilon_0(z_0)$ 
sufficiently small. Then\footnote{We slightly corrected and extended the formulation of the remainder of this theorem.} $K(\dott)$ is analytic on $\rho(H_0)$ and finitely meromorphic on $D(z_0; \varepsilon_0)$. In particular, near $z_0$, $K(\dott)$ takes on the form  
\begin{equation}
K(z) = \sum_{k=-N_0}^{\infty} (z-z_0)^k K_k(z_0), \quad 
0 < |z - z_0| < \varepsilon_0,    \lb{4.9a}
\end{equation}
for some $N_0 = N_0(z_0) \in \bbN$ and some $0 < \varepsilon_0$ 
sufficiently small, with  
\begin{equation}\lb{4.9b} 
K_{-k}(z_0) \in \cF(\cK), \; 1 \leq k \leq N_0(z_0), \quad 
K_k(z_0) \in \cB(\cK), \; k \in \bbN \cup \{0\}.   
\end{equation} 
Given \eqref{4.9a}, we now assume that\footnote{Condition \eqref{4.9c} was inadvertently omitted 
in our paper,  {\it Integral Eq.~Operator Theory} {\bf 82}, 61--94 (2015), as kindly pointed out by Jussi Behrndt (see the erratum in IEOT {\bf 85}, 301--302 (2016)).} 
\begin{equation}
[I_{\cK} - K_0(z_0)] \in \Phi(\cK).    \lb{4.9c} 
\end{equation} 
Then $z_0$ is a discrete eigenvalue of $H$,  
\begin{equation}
z_0 \in \sigma_d(H),     \lb{4.10}
\end{equation}
and $I_{\cK} - K(\cdot)$ satisfies the conditions of Hypothesis \ref{h4.3} on 
$\Omega= D(z_0; \varepsilon_1)$ for $0 < \varepsilon_1 < \varepsilon_0$.  
In addition,  
\begin{equation}
m_a(z_0; H) = m_a(z_0; H_0)  + \ind_{C(z_0; \varepsilon)}(I_{\cK} - K(\dott)), 
\quad 0 < \varepsilon < \varepsilon_1.      \lb{4.11}
\end{equation}
\end{theorem} 
\begin{proof}\footnote{We corrected and extended the first paragraph of this proof.}
That \eqref{4.10} holds is shown as follows: The assumption $z_0 \in \sigma_d (H_0)$ yields that 
$R_0 (\dott)$, $\ol{R_0(\dott) V_2^*}$, $V_1 R_0(\dott)$, $K(\dott)$, and $I_{\cK} - K(\dott)$ are analytic in 
$D(z_0;\varepsilon_0)\backslash \{z_0\}$ and finitely meromorphic on $D(z_0;\varepsilon_0)$. 
The assumption $D(z_0; \varepsilon_0) \cap \sigma(H) = \{z_0\}$ yields 
$D(z_0;\varepsilon_0)\backslash \{z_0\} \subset \rho(H)$  
and hence \eqref{2.25} implies that $[I_{\cK} - K(\dott)]^{-1} \in \cB(\cK)$ on 
$D(z_0;\varepsilon_0)\backslash \{z_0\}$. Thus, applying the meromorphic Fredholm Theorem \ref{t4.2}, 
$[I_{\cK} - K(\cdot)]^{-1}$ is analytic in $D(z_0;\varepsilon_1)\backslash \{z_0\}$ and finitely meromorphic on $D(z_0;\varepsilon_1)$ for some $0 < \varepsilon_1 < \varepsilon_0$. In particular, $[I_{\cK} - K(\cdot)]^{-1}$ satisfies the conditions of Hypothesis \ref{h4.3}. By \eqref{2.13}, this implies that $R(\dott)$ analytic in $D(z_0;\varepsilon_1)\backslash \{z_0\}$ and finitely meromorphic on $D(z_0;\varepsilon_1)$. As in the proof of Theorem \ref{t3.11},  \cite[Sect.\ III.6.5]{Ka80} implies that $z_0 \in \sigma_p(H)$ and then again by the finitely meromorphic property of $R(\cdot)$ on $D(z_0;\varepsilon_1)$, the Riesz projection associated with $z_0$, is finite-dimensional, which in turn is equivalent to the eigenvalue $z_0$ having finite algebraic multiplicity and hence yields \eqref{4.10}. 

The rest also follows the proof of Theorem \ref{t3.11}: 
\begin{align}
& \ind_{C(z_0; \varepsilon)}(I_{\cK} - K(\cdot)) = \f{1}{2 \pi i} {\tr}_{\cK}\bigg( 
\ointctrclockwise_{C(z_0; \varepsilon)} d\zeta \, 
[I_{\cK} - K(\zeta)]^{-1}[- K'(\zeta)]\bigg)     \no \\
& \quad = \f{1}{2 \pi i} {\tr}_{\cK}\bigg( 
\ointctrclockwise_{C(z_0; \varepsilon)} d\zeta \, 
[I_{\cK} - K(\zeta)]^{-1} V_1 R_0(\zeta) \ol{R_0(\zeta) V_2^*}\bigg)  \no \\ 
& \quad = \f{1}{2 \pi i} {\tr}_{\cH}\bigg( 
\ointctrclockwise_{C(z_0; \varepsilon)} d\zeta \, 
\ol{R_0(\zeta) V_2^*} [I_{\cK} - K(\zeta)]^{-1} V_1 R_0(\zeta)\bigg)    \no \\
& \quad = \f{-1}{2 \pi i} {\tr}_{\cH} \bigg(
\ointctrclockwise_{C(z_0; \varepsilon)} d\zeta \, 
\big[(H - \zeta I_{\cH})^{-1} - (H_0 - \zeta I_{\cH})^{-1}\big] \bigg)    \no \\ 
& \quad = m_a(z_0;H) - m_a(z_0; H_0).        \lb{4.12} 
\end{align} 
\end{proof}

\section{Pairs of Projections and an Index Computation}
\lb{s6}

In our final section we apply some of the principal results of Section \ref{s4} and \ref{s5} to 
pairs of projections and their index and make a connection to an underlying 
Birman--Schwinger-type operator.

One recalls that an ordered pair of projections $(P,Q)$ is called a {\it Fredholm pair} if 
\begin{equation} 
QP:\ran(P)\rightarrow \ran(Q) \, \text{ belongs to $\Phi(\cH)$.} 
\end{equation} 
In this case, the {\it index} of the 
pair $(P,Q)$, denoted $\ind(P,Q)$, is defined to be the index of the Fredholm operator $QP$:
\begin{equation}\lb{7.1}
\ind(P,Q) = \dim(\ker(QP)) - \dim(\ker((QP)^*)).
\end{equation}
For pertinent literature related to pairs of projections, we refer to \cite{AS94}, \cite{ASS94}, 
\cite{BS10}, \cite{Da58}, \cite{Ef89}, \cite{Ha69}, \cite{Ka97}, \cite{Ka55}, \cite{Pu11}, \cite{Si10}, and the 
references cited therein.  The following results are well-known and will be used throughout this section.

\begin{theorem}[\cite{AS94}, \cite{ASS94}]\lb{t7.1}
If $(P,Q)$ is a pair of orthogonal projections, then the following items hold.\\[1mm]
$(i)$  If $(P-Q) \in \cB_{\infty}(\cH)$, then $(P,Q)$ is a Fredholm pair.\\[1mm]
$(ii)$  If $(P,Q)$ is a Fredholm pair, then
\begin{equation}\lb{7.2}
\ind(P,Q) = m_1-m_{-1},
\end{equation}
where
\begin{equation}\lb{7.3}
\begin{split}
m_1&:= \dim(\{v\in \cH\, |\, Pv=v,\, Qv=0\}),\\
m_{-1}&:=\dim(\{v\in \cH\, |\, Pv=0,\, Qv=v\}).
\end{split}
\end{equation}
$(iii)$  If $(P,Q)$ is a Fredholm pair, then $(Q,P)$ is a Fredholm pair, and
\begin{equation}\lb{7.3aa}
\ind(Q,P) = -\ind(P,Q).
\end{equation}
$(iv)$  If $(P,Q)$ is a Fredholm pair with $(P-Q) \in \cB_{2n+1}(\cH)$, for some $n\in \bbN$, then
\begin{equation}\lb{7.3a}
\tr_{\cH}\big((P-Q)^{2n+1}\big) = \ind(P,Q).
\end{equation}
\end{theorem}

If $(P,Q)$ is a pair of orthogonal projections such that $(P-Q) \in \cB_1(\cH)$, then the 
associated {\it perturbation determinant},
\begin{align}
\begin{split} 
& D_{(P,Q)}(z) = \text{det}_{\cH}\big(I_{\cH} + (P-Q)(Q-zI_{\cH})^{-1} \big)     \lb{7.6a} \\
& \hspace*{1.45cm} = \text{det}_{\cH}\big((P-zI_{\cH})(Q-zI_{\cH})^{-1} \big),  \quad 
 z\in \bbC\backslash \{0,1\},   
 \end{split} 
\end{align}
is well-defined.  The {\it Krein--Lifshitz spectral shift function} $\xi_{(P,Q)}(\,\cdot\,)$ corresponding to the pair $(P,Q)$ is then given by
\begin{equation}\lb{7.7a}
\xi_{(P,Q)}(\lambda) = \lim_{\varepsilon \downarrow 0}\frac{1}{\pi} \arg [D_{(P,Q)}(\lambda+i\varepsilon)]\, 
\text{ for a.e.\ $\lambda \in \bbR$} 
\end{equation}
(w.r.t. Lebesgue measure), and $\xi_{(P,Q)}$ satisfies:
\begin{align}
\xi_{(P,Q)} \in L^1(\bbR),\quad \int_{\bbR} |\xi_{(P,Q)}(\lambda)|\, d\lambda \leq \|P-Q\|_{\cB_1(\cH)}.\lb{7.8a}
\end{align}
Moreover, integrating the spectral shift function allows one to recover the trace of $P-Q$, that is, 
\begin{equation}\lb{7.9a}
\int_{\bbR} \xi_{(P,Q)}(\lambda)\, d\lambda = \tr_{\cH}(P-Q).
\end{equation}
The properties in \eqref{7.8a} and \eqref{7.9a} follow from general considerations and do not rely on the fact that $P$ and $Q$ are orthogonal projections.  For details on the Krein--Lifshitz spectral shift function within the general context of trace class (and, more generally, resolvent comparable) perturbations of self-adjoint operators, we refer to \cite{BY93}, 
\cite[Ch.\ 8]{Ya92}, \cite{Ya07}, \cite[Sect.\ 0.9, Chs.\ 4, 5, 9]{Ya10}.  In the specific case at hand, where $P$ and $Q$ are orthogonal projections, the perturbation determinant and spectral shift function may be computed explicitly.

\begin{theorem}[\cite{AS94}]\lb{thm7.2}
Suppose $(P,Q)$ is a pair of orthogonal projections with $(P-Q) \in \cB_1(\cH)$. Then the following items hold.\\[1mm]
$(i)$ The perturbation determinant $D_{(P,Q)}(\cdot)$ is given by
\begin{align}
D_{(P,Q)}(z) = \bigg(\frac{z-1}{z} \bigg)^{m_1-m_{-1}},\quad z\in \bbC\backslash [0,1],\lb{7.10a}
\end{align}
and the spectral shift function is piecewise constant,
\begin{align}
\xi_{(P,Q)}(\lambda) = 
\begin{cases}
0,& \lambda \notin[0,1], \\
m_1-m_{-1},& \lambda \in [0,1],
\end{cases}\, \text{ for a.e.\ $\lambda\in \bbR$},      \lb{7.11a}
\end{align}
where $m_{\pm 1}$ are defined in \eqref{7.3}.\\[1mm]
$(ii)$ If $f$ is a measurable, complex-valued function which is continuous in neighborhoods of $z=0$ and $z=1$, then $[f(P)-f(Q)] \in \cB_1(\cH)$ and 
\begin{equation}\lb{7.12a}
\tr_{\cH}[f(P)-f(Q)] = [f(1)-f(0)]\tr_{\cH}(P-Q) = [f(1)-f(0)](m_1-m_{-1}).
\end{equation}
\end{theorem}

Finally, we compute the index of an operator $M_{(P,Q)}(\cdot)$ closely related to the 
Birman--Schwinger-type operator naturally associated with the pair of projections $(P,Q)$.

\begin{theorem}\lb{thm7.3}
Suppose that $(P,Q)$ is a pair of orthogonal projections with $(P-Q) \in \cB_1(\cH)$ and that 
$M_{(P,Q)}(\,\cdot\,):\bbC\backslash \{0,1\}\rightarrow \cB(\cH)$ is given by
\begin{align}  \lb{7.4}
\begin{split} 
M_{(P,Q)}(z):= (P-zI_{\cH})(Q-zI_{\cH})^{-1} = I_{\cH} + (P - Q)(Q-zI_{\cH})^{-1},& \\ 
z\in \bbC\backslash \{0,1\}.&
\end{split} 
\end{align}
Then the following items hold.\\[1mm]
$(i)$ $(P,Q)$ is a Fredholm pair.\\[1mm]
$(ii)$ $M_{(P,Q)} (\cdot)$ is finitely meromorphic, and $M_{(P,Q)}(z)\in \Phi(\cH)$, 
$z\in \bbC\backslash \{0,1\}$.\\[1mm]
$(iii)$ The index of $M_{(P,Q)}(\cdot)$ with respect to the counterclockwise oriented circle 
$C(z_0;\varepsilon)$, with $\varepsilon>0$ taken sufficiently small, is given by
\begin{equation}\lb{7.5}
\ind_{C(z_0;\varepsilon)} (M_{(P,Q)}(\cdot)) =
\begin{cases}
- \ind(P,Q),& z_0=0, \\
0,& z_0\in \bbC\backslash\{0,1\}, \\
\ind(P,Q),& z_0=1. 
\end{cases}  
\end{equation}
\end{theorem}
\begin{proof}
Item $(i)$ follows immediately from Theorem \ref{t7.1}, owing to the assumption that 
$(P-Q)\in \cB_1(\cH)$.  The claims in item $(ii)$ follow from the representation
\begin{equation}\lb{7.6}
M_{(P,Q)} (z) = I_{\cH} + (P-Q)(Q-zI_{\cH})^{-1}, \quad z\in \bbC\backslash \{0,1\},
\end{equation}
which may be obtained by applying a standard resolvent identity in \eqref{7.4}, and the fact 
that $(I_{\cH}-K) \in \Phi(\cH)$ if $K\in \cB_{\infty}(\cH)$.  It remains to settle item $(iii)$.  To this end, 
one computes
\begin{align}
M_{(P,Q)} (z)^{-1} = (Q-zI_{\cH})(P-zI_{\cH})^{-1}, \quad z\in \bbC\backslash \{0,1\}.\lb{7.8}
\end{align}
and
\begin{align}\lb{7.9}
M_{(P,Q)}'(z)=(P-Q)(Q-zI_{\cH})^{-2},\quad z\in \bbC\backslash \{0,1\}.
\end{align}
As a result, 
\begin{align}\lb{7.10}
M_{(P,Q)}'(z) M_{(P,Q)}(z)^{-1} = (P-Q)(Q-zI_{\cH})^{-1}(P-zI_{\cH})^{-1},\quad z\in \bbC\backslash \{0,1\},
\end{align}
and one computes, for $\varepsilon>0$ sufficiently small,
\begin{align}
\ind_{C(z_0;\varepsilon)}(M_{(P,Q)}(\cdot))
&= \tr_{\cH}\bigg\{ \frac{1}{2\pi i}\ointctrclockwise_{C(z_0; \varepsilon)} 
dz \, M_{(P,Q)}'(z) M_{(P,Q)}(z)^{-1}\bigg\}\no\\
&= \tr_{\cH}\bigg\{ \frac{1}{2\pi i}\ointctrclockwise_{C(z_0; \varepsilon)} dz \, (P-Q)(Q-zI_{\cH})^{-1}(P-zI_{\cH})^{-1} \bigg\}\no\\
& = -\tr_{\cH}\bigg\{ \frac{1}{2\pi i}\ointctrclockwise_{C(z_0; \varepsilon)} dz \, (P-zI_{\cH})^{-1}(Q-P)(Q-zI_{\cH})^{-1} \bigg\}\no\\
& = \tr_{\cH}\bigg\{ \frac{1}{2\pi i}\ointctrclockwise_{C(z_0; \varepsilon)} dz \, \big[(Q-zI_{\cH})^{-1}-(P-zI_{\cH})^{-1}\big] \bigg\}\no\\
& = \frac{1}{2\pi i}\ointctrclockwise_{C(z_0; \varepsilon)} dz \, 
\tr_{\cH}\big[(Q-zI_{\cH})^{-1}-(P-zI_{\cH})^{-1}\big]  \no\\
& = \frac{\ind(P,Q)}{2\pi i}\ointctrclockwise_{C(z_0; \varepsilon)} dz \, \bigg(\frac{1}{z-1}-\frac{1}{z} \bigg).  \lb{7.11}
\end{align}
To obtain the last equality in \eqref{7.11} one applies \eqref{7.12a}.  Finally, \eqref{7.5} follows from \eqref{7.11}, \eqref{7.3aa}, and the residue calculus.
\end{proof}

We conclude by noting that the actual Birman--Schwinger operator associated with the pair 
$(P,Q)$ is then given by 
$K_{(P,Q)}(z):= - (P - Q)(Q-zI_{\cH})^{-1}$, and hence 
\begin{equation} 
M_{(P,Q)}(z) = I_{\cH} - K_{(P,Q)}(z) = I_{\cH} + (P - Q)(Q - z I_{\cH})^{-1}, 
\quad z \in \bbC \backslash \{0,1\}. 
\end{equation}

\medskip
\noindent {\bf Acknowledgments.} 
We are indebted to Yuri Latushkin and Alim Sukhtayev for discussions. We also 
thank the anonymous referee for very helpful suggestions and, particularly, for suggesting 
the current, shorter version of the proofs of Theorems~\ref{t3.7} and \ref{t3.8}. F.G. gratefully 
acknowledges the extraordinary hospitality of the Department of Mathematical Sciences of the Norwegian University of Science and Technology, Trondheim, during a series of visits. 


\end{document}